\documentclass[a4paper,UKenglish]{scrartcl}

\DeclareRobustCommand{\authorthing}{Meike Hatzel\thanks{Meike Hatzel's research was supported by the Federal Ministry of Education and Research (BMBF) and by a fellowship within the IFI programme of the German Academic Exchange Service (DAAD).}
\and
Johannes Schröder
}
\author{\authorthing}

\newif\ifcomment
\commenttrue
\commentfalse

\newif\iflong
\newif\ifshort

\longtrue
\iflong
\else
\shorttrue
\fi

\usepackage{ifthen}
\usepackage{amssymb}
\usepackage{amsthm}
\usepackage{microtype}
\usepackage{amsmath}
\usepackage{amssymb}
\usepackage{amsfonts}
\usepackage{mathtools}
\usepackage{xcolor}
\usepackage{xstring}
\usepackage{marginnote} 

\usepackage{hyperref}
\usepackage{cleveref}
\usepackage{thm-restate}
\usepackage{macros}

\usepackage{rotating}

\usepackage{xcolor,graphicx}
\usepackage{caption,tabularx,booktabs,multirow}
\usepackage[margin=0pt]{subcaption}
\usepackage[skip=5pt]{parskip}
\usepackage{tikz}
\usepackage{float}
\usepackage{pdflscape}
\usepackage{afterpage}
\usepackage{capt-of}
\usepackage{color}									
\usepackage{colortbl}
\usepackage{epsfig}
\usepackage{verbatim}
\usepackage{csquotes}
\usepackage{mathrsfs}
\usepackage{dsfont}
\usepackage{xifthen}
\usepackage[inline]{enumitem}
\usepackage{algorithm2e}
\usepackage{algorithmic}
\usepackage{glossaries}
\usepackage{pgfplots}
\usepackage{comment}
\pgfplotsset{compat=1.18}
\allowdisplaybreaks

\newcounter{definitionCounter}
\newenvironment{enamerate}[4][]
{
	\begin{enumerate}[#1,noitemsep,topsep=0pt,label=\textbf{(#2\arabic*)},labelindent=0pt,labelwidth=\widthof{\ref{#3}},itemindent=#4%
		]%
	}
	{\end{enumerate}}

\newenvironment{enameratesmall}[4][]
{
	\begin{enumerate*}[#1,noitemsep,topsep=0pt,label=\textbf{(#2\arabic*)},labelindent=0pt,labelwidth=\widthof{\ref{#3}},itemindent=#4%
		]%
	}
	{\end{enumerate*}}

\newenvironment{customlegend}[1][]{%
    \begingroup
    \csname pgfplots@init@cleared@structures\endcsname
    \pgfplotsset{#1}%
}{%
    \csname pgfplots@createlegend\endcsname
    \endgroup
}%

\def\addlegendimage{\csname pgfplots@addlegendimage\endcsname}

\usetikzlibrary{positioning,backgrounds,patterns,calc,shapes}
\DeclareMathAlphabet{\pazocal}{OMS}{zplm}{m}{n}
\tikzstyle{para}=[rectangle,draw=black,minimum height=.8cm,fill=gray!10,rounded corners=1mm, on grid]

\usetikzlibrary{arrows,decorations.pathmorphing,decorations.pathreplacing,backgrounds,positioning,fit,matrix}
\usetikzlibrary{shapes,calc,fadings}
\usepackage{todonotes}

\usepackage{tikz}
\usetikzlibrary{arrows.meta,decorations.pathmorphing,decorations.pathreplacing,backgrounds,positioning,fit,matrix}
\usetikzlibrary{shapes,calc}
\usetikzlibrary{plotmarks}
\pgfkeys{/pgfplots/number in legend/.style={%
        /pgfplots/legend image code/.code={%
            \node at (0.125,-0.0225){#1}; 
        },%
    },
}
\pgfplotsset{
every legend to name picture/.style={west}
}

\pgfdeclarelayer{background}
\pgfdeclarelayer{foreground}
\pgfsetlayers{background,main,foreground}

\pgfkeys{%
	/tikz/on layer/.code={
		\pgfonlayer{#1}\begingroup
		\aftergroup\endpgfonlayer
		\aftergroup\endgroup
	},
	/tikz/node on layer/.code={
		\pgfonlayer{#1}\begingroup
		\expandafter\def\expandafter\tikz@node@finish\expandafter{\expandafter\endgroup\expandafter\endpgfonlayer\tikz@node@finish}%
	}
}

\usetikzlibrary{shapes.geometric,hobby,calc} 
\usetikzlibrary{decorations.pathmorphing,decorations.pathreplacing}
\usetikzlibrary{decorations.markings,intersections}
\usetikzlibrary{decorations.text}
\usetikzlibrary{shapes.misc}
\usetikzlibrary{decorations,shapes}
\usetikzlibrary{patterns,intersections,through,backgrounds}

\usepackage{xspace}
\usepackage{dsfont}
\usepackage{soul}

\usepackage{stackengine}
\stackMath

\renewcommand{\emptyset}{\varnothing}

\hypersetup{
	colorlinks=true,
	linkcolor=myBlue,
	citecolor=myBlue,
	urlcolor=myBlue,
	bookmarksopen=true,
	bookmarksnumbered,
	bookmarksopenlevel=2,
	bookmarksdepth=3
}

\title{Computing $\Separations$-DAGs and Parity Games} 

\usepackage{textpos}

\newcommand{\appsymb}{$\star$}
\newcommand{\appref}[1]{\ifshort{}[{\hyperref[proof:#1]{\appsymb}}]\fi{}}

\newcommand{\toappendix}[1]{%
    \iflong{}{#1}\else{}\gappto{\appendixProofText}{{#1}}\fi{}
}

\newcommand{\appendixproof}[2]{%
    \iflong{}{#2}\else\gappto{\appendixProofText}{\subsection{Proof of \cref{#1}}\label{proof:#1}#2}\fi{}
}

\newcommand{\appendixsection}[1]{%
    \iflong{}{}\else\gappto{\appendixProofText}{\section{Additional Material for \cref{#1}}}\fi{}
}

\begin{document}
	\maketitle
	
	\begin{textblock}{20}(-1.9,8.3)
        \includegraphics[width=80px]{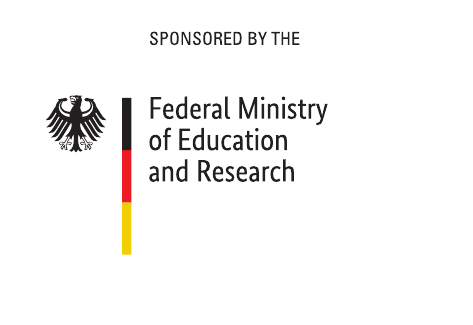}%
    \end{textblock}
	
	\begin{abstract}
        Treewidth on undirected graphs is known to have many algorithmic applications.
        When considering directed width-measures there are much less results on their deployment for algorithmic results.
        In 2022 the first author, Rabinovich and Wiederrecht introduced a new directed width measure, $\Separations{}$-DAG-width, using directed separations and obtained a structural duality for it.
        In 2012 Berwanger~et~al.~solved Parity Games in polynomial time on digraphs of bounded DAG-width.
        With generalising this result to digraphs of bounded $\Separations{}$-DAG-width and also providing an algorithm to compute the $\Separations{}$-DAG-width of a given digraphs we give first algorithmical results for this new parameter.
	\end{abstract}
\newpage
\section{Introduction}

Treewidth was popularized by Robertson and Seymour~\cite{robertson1984graph} and since then a lot of influential work has been grounded in the understanding of the characteristics of graphs based on their treewidth.
Treewidth has numerous applications both in structure theory as well as algorithmical, as it established itself as indispensable for structural graph theory throughout the graph minor project\cite{robertson1986graph} and beyond\cite{oporowski1993typical} as well as leading to algorithmic applications~\cite{arnborg1987complexity,arnborg1989linear} including even a powerful meta theorem discovered by Courcelle~\cite{courcelle1990monadic}.
Subsequently numerous concepts for decompositions of graphs and even of other structures have been considered~\cite{espelage2001solve,oum2005rank,kaminski2009recent,dallard2021treewidth,bonnet2021twin,merlin2021spined}.

Naturally, many attempts have been made to try to find equally useful parameters for directed graphs encompassing both similarly powerful algorithmic properties to treewidth as well as structural properties by establishing a duality with high connectivity areas in the graph\cite{robertson1986graph}.
These attempts include directed treewidth~\cite{reed1999introducing,johnson2001directed}, Kelly-width\cite{hunter2008digraph} and DAG-width\cite{berwanger2006dag}.
However none of these parameters have so far gotten the same level of traction as treewidth, partially because of their comparative shortcomings when it comes to algorithmic applications.
On the contrary, deciding all $\variablestyle{MSO}_1$ properties efficiently on digraphs under a bounded parameter, as has been demonstrated for treewidth~\cite{courcelle1990monadic}, can not be achieved by any parameter closed under directed topological minors and independent of (undirected) treewidth~\cite{ganian2016there}.
Still, this result leaves some space for possible algorithmic applicability for directed parameter and indeed some algorithmic results have been found for width parameters on directed graphs such as Hamiltonian path being solvable in polynomial time on graphs of bounded directed treewidth~\cite{johnson2001directed} and parity games being solvable in polynomial time on graphs of bounded DAG-width~\cite{berwanger2006dag}. 

Many of the known digraph parameters are related to some version of the cops and robber game and by extension to graph searching similar to treewidth.
The DAG-width of a graph in particular is equal to the number of cops for which there exists a monotone strategy to catch a mobile robber who can move along any path in the graph not occupied by a cop who does not switch positions, this is called the \emph{cops-and-robber-reachability-game}.
The number of cops needed to catch the robber without relying on monotonicity however can be significantly lower as shown in~\cite{kreutzer2011digraph} and it is an open problem to find a bound on the monotonicity costs for this cops and robber game on directed graphs.
However, recently a new graph decomposition including a corresponding width parameter for directed graphs has been proposed, namely $\Separations$-DAGs and $\Separations$-DAG-width~\cite{dissMeike}.
This parameter is more general than DAG-width and in particular displays a duality to non-monotone cop strategies in the cops and robber game, while also having a duality with a tangle analogue for directed graphs.
As the parameter is very new though, useful algorithmic properties of $\Separations$-DAG-width are yet to be discovered.

In this paper we consider the algorithmic applications of the parameter $\Separations$-DAG-width for the first time.
As $\Separations$-DAG-width lies between the two established directed width measures directed treewidth and DAG-width in the sense that every class of digraphs with bounded DAG-width has bounded $\Separations$-DAG-width and every class of digraphs with bounded $\Separations$-DAG-width has bounded directed treewidth,
the relevant problems to explore for $\Separations$-DAG-width are those, which cannot be solved efficiently for digraphs of bounded directed treewidth, but are not known to be hard on digraphs of bounded DAG-width.

To actually harness algorithmic power of $\Separations$-DAGs, we first need an algorithm to efficiently compute them.
We tackle this problem in~\cref{sec:ConstructSDAG} and prove the following.
\begin{restatable}{proposition}{compSDAG}
	\label{compSDAG}
    Given a a digraph $D$ of $\Separations$-DAG-width $k,$ an $\Separations$-DAG of minimum width for $D$ can be computed in $O(|V(D)|^{h(k)})$ for a linear function $h.$
\end{restatable}
Subsequently, in~\cref{sec:ParGames}, we provide an adaptation of the high profile algorithm from~\cite{berwanger2006dag}, which solves parity games on DAG decompositions, to $\Separations$-DAGs obtaining our main result.
\begin{restatable}{proposition}{solveParity}
    \label{solveParity}
    Every parity game of bounded $\Separations$-DAG-width can be solved in polynomial time.
\end{restatable}

\ifshort{Due to space constraints we moved some technical proofs and details into the appendix while keeping proofsketches for the essential ideas.
Results with additional details in the appendix are marked by (\appsymb).}{}\else{}\fi

\section{Preliminaries}
\label{sec:prelem}
\appendixsection{sec:prelem}
\newcommand*{\tail}[1]{\Fkt{\variablestyle{tail}}{#1}}

All digraphs in this paper are simple, finite and without loops.
We refer the reader to~\cite{bang2008digraphs} for any standard definitions.
We use $\mathds{N}$ for the set of all non-negative integers.
For $i,j \in \mathds{N},$ we use $[i,j]$ to refer to the set of all integers which are greater or equal to $i$ and smaller or equal to $j.$
For a set $S,$ let $[S]$ be the set of all subsets of $S,$ let $[S]^{=i}$ be the set of all subsets of $S$ of size exactly $i$ and let $[S]^{\leq i}$ be the set of all subsets of $S$ of size at most $i.$
Similarly, let $S^{\mathds{N}}$ be the set of all finite tuples over $S$ and let $S^{i}$ be the set of all such tuples of length $i$.
For two sets $S_1$ and $S_2$ we denote the symmetric difference between $S_1$ and $S_2$ by $S_1 \triangle S_2$.
\toappendix{
Furthermore, we use $t_1:t_2$ to indicate the concatenation of two sequences $t_1$ and $t_2.$
For a sequence $t,$ let $\tail{t}$ be the last element in $t.$
We also index the elements in $t$ with $1$ to $|t|$ such that $t[i]$ is the $i$-th element in $t$ and $t[i,j]$ is the subsequence of $t$ containing the $i$-th to $j$-th element in $t.$}


\subsection{Digraphs}
A digraph $D=(V,E \subseteq (V \times V))$ is a tuple of its\define{vertex set} $V(D) \coloneqq V$ and its\define{edge set} $E(D) \coloneqq E.$
We define $|D|\coloneqq |V(D)| + |E(D)|$ to be the\define{size of $D$}. 
For $v \in V(D)$ we call $\outN{D}{v} \coloneqq \{w \in V(D) | (v,w) \in E(D) \}$ the\define{out-neighbourhood of $v$ in $D$} and $\inN{D}{v} \coloneqq \{w \in V(D) | (w,v) \in E(D) \}$ the\define{in-neighbourhood of $v$ in $D$}.
We also extend this notion to subsets $V' \subseteq V(D)$ by defining $\outN{D}{V'} \coloneqq (\bigcup_{v\in V'} \outN{D}{v}) \setminus V'$ and $\inN{D}{V'} \coloneqq (\bigcup_{v\in V'} \inN{D}{v})\setminus V'.$
Similarly, we define the\define{out-degree of $v$} by $\outD{D}{v} \coloneqq |\outN{D}{v}|$ and the\define{in-degree of $v$} by $\inD{D}{v} \coloneqq |\inN{D}{v}|.$
We call a digraph $D'$ a\define{subgraph of $D$}, if $V(D') \subseteq V(D)$ and $E(D') \subseteq E(D).$
For $V' \subseteq V(D),$ let $D[V'] \coloneqq (V',\{(v,w) |  v,w \in V' \text{ and } (v,w) \in E(D)\})$ be the\define{induced subgraph of $D$ on the vertex set $V'$}.
We use $D \setminus V'$ to denote the digraph $D[V(D) \setminus V'].$ 

For $l \in \mathds{N},$ we call a sequence $P = (v_0,v_1, \dots , v_{l-1}, v_{l}) \in V(D)^{\mathds{N}}$ a\define{path of length $l$ in $D$ from $v_0$ to $v_l$}, if $(v_i, v_{i+1}) \in E(D)$ for all $i \in [0,i-1].$
We call such a sequence $P$ a\define{circle of length $l+1$ in $D$}, if it is a path of length $l$ in $D$ and $(v_l,v_0) \in E(D).$ 
Let $v \in V(D).$
We use $\Reach{D}(v)$ to denote the set of vertices $w$ for which there exists a path from $v$ to $w$ in $D$ and we extend this notation to vertex sets $V' \subseteq V(D)$ using $\Reach{D}(V') \coloneqq \bigcup_{v \in V'} \Reach{D}(v).$
We call these sets the\define{reach of $v$} and\define{reach of $V'$} respectively.
We say $D$ is a\define{DAG} (directed acyclic graph) if there are no directed cycles in $D.$
For a DAG $D$ and $v \in V(D),$ we call $\outN{D}{v}$ the \define{children} of $v$ in $D.$
We say $D$ is a\define{directed tree} if there exists a vertex $v \in V(D),$ such that for every vertex $w \in V(D)$ there exists a unique path from $v$ to $w$ in $D.$
We further define $\preccurlyeq_D$ as the minimal partial ordering on the vertices of the DAG $D$ such that $v \preccurlyeq_D w$ for all $v,w \in V(D)$ for which there exists a path in $D$ from $v$ to $w.$
We say $v$ is\define{above} $w$ in $D$ and $w$ is\define{below} $v.$

\subsection{Directed separations}

For $S_A,S_B \subseteq V(D),$ we call $\mathscr{S} = (S_A \rightarrow S_B)$ a\define{separation in $D$} if $V(D) = S_A \cup S_B$ and $(b,a) \notin E(D)$ for all $a \in S_A \setminus S_B, b \in S_B \setminus S_A.$
We define $\Separations_D$ to be the set of separations in $D.$
We call $V_{\top}(\mathscr{S}) \coloneqq S_A \setminus S_B$ the\define{top side of $\mathscr{S}$}, $V_{S}(\mathscr{S}) \coloneqq S_A \cap S_B$ the\define{separator of $\mathscr{S}$} and $V_{\bot}(\mathscr{S}) \coloneqq S_B \setminus S_A$ the\define{bottom side of $\mathscr{S}$} for $\mathscr{S} \in \Separations_D.$ 
We also define $V_{A}(\mathscr{S}) \coloneqq S_A$ and $V_{B}(\mathscr{S}) \coloneqq S_B.$
We call $\mathscr{S}$ a\define{separation of order $k$ in $D$} if $|V_{S}(\mathscr{S})| = k$ and we define $\Separations_D^k$ to be the set of separations of order at most $k.$

\begin{figure}[!ht]
\centering
\begin{tikzpicture}[scale=0.7,>={Stealth[width=3mm,length=2mm]}]
    \begin{customlegend}[
        legend entries={ 
            $V_{\top}$,
            $V_A$,
            $V_{S}$,
            $V_B$,
            $V_{\bot}$
        },
        legend style={at={(8,1.5)},font=\footnotesize}] 
        \addlegendimage{area legend,fill=lightred,draw opacity=0}
        \addlegendimage{number in legend={$\bigcirc$},color=darkred}
        \addlegendimage{area legend,fill=customgreen,draw opacity=0}
        \addlegendimage{number in legend={$\bigcirc$},color=darkblue}
        \addlegendimage{area legend,fill=lightblue,draw opacity=0}
    \end{customlegend}
    \def\firstcircle{(-1.5,0) circle (2cm)}
  \def\secondcircle{(1.5,0) circle (2cm)}
    \begin{scope}
    \clip \firstcircle;
    \fill[customgreen] \secondcircle;
      \end{scope}
      \begin{scope}
    \clip \secondcircle;
    \fill[color=lightblue, even odd rule] \firstcircle \secondcircle;
      \end{scope}
      \begin{scope}
    \clip \firstcircle;
    \fill[color=lightred, even odd rule] \firstcircle \secondcircle;
      \end{scope}
      \draw[color=darkred] \firstcircle node {$V_A$};
      \draw[color=darkblue] \secondcircle node {$V_B$};
      \draw (-2,1.5) edge[->,bend left=30] (2,1.5);
      \draw (2,-1.5) edge[->,bend left=30] node[midway] (x) {} (-2,-1.5);
      \draw[color=red,line width=1.3pt] ($(x)+(45:0.4)$) edge ($(x)+(225:0.4)$);
      \draw[color=red,line width=1.3pt] ($(x)+(135:0.4)$) edge ($(x)+(315:0.4)$);
\end{tikzpicture}
    \caption{This figure illustrates the structure of the sets making up a separation. The arrows indicate that there can be edges from $V_{\top}$ to $V_{\bot}$ but not in the other direction.}
    \label{fig:SeparationExp}
\end{figure}
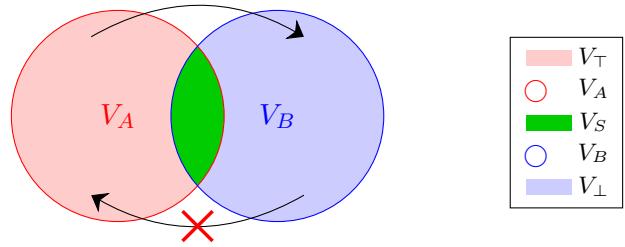

Furthermore, we define a partial order $\leq$ on $\Separations_D$ such that $\mathscr{S} \leq \mathscr{S}',$ if $V_A(\mathscr{S}) \subseteq V_A(\mathscr{S}')$ and $V_B(\mathscr{S}) \supseteq V_B(\mathscr{S}').$
Let $<$ be the set of tuples in $\leq$ that are not symmetric.
We say $\mathscr{S}$ is\define{laminar} to $\mathscr{S}'$ if $\mathscr{S} \leq \mathscr{S}'$ or $\mathscr{S}' \leq \mathscr{S}.$
If two separations are not laminar, they\define{cross}. 
For two separations $\mathscr{S}_1 \coloneqq (S_{1,A} \rightarrow S_{1,B}), \mathscr{S}_2 \coloneqq (S_{2,A} \rightarrow S_{2,B}),$ we can create two new separations called\define{uncrossings}.
The separation $(S_{1,A} \cap S_{2,A} \rightarrow S_{1,B} \cup S_{2,B})$ is called the\define{lower uncrossing of $\mathscr{S}_1$ and $\mathscr{S}_2$}, also written as $\mathscr{S}_1 \wedge \mathscr{S}_2,$ while the separation $(S_{1,A} \cup S_{2,A} \rightarrow S_{1,B} \cap S_{2,B})$ is called the\define{upper uncrossing of $\mathscr{S}_1$ and $\mathscr{S}_2$} also written as $\mathscr{S}_1 \vee \mathscr{S}_2.$
This notion extends to sets $Z \subseteq \Separations_D$ by the lower uncrossing $\bigwedge Z \coloneqq (\bigcap_{\mathscr{S} \in Z} V_A(\mathscr{S}) \rightarrow \bigcup_{\mathscr{S} \in Z} V_B(\mathscr{S}))$ and the upper uncrossing $\bigvee Z \coloneqq (\bigcup_{\mathscr{S} \in Z} V_A(\mathscr{S}) \rightarrow \bigcap_{\mathscr{S} \in Z} V_B(\mathscr{S})).$

\begin{definition}
    \label{sepDef}
    For $X \subseteq V(D) \setminus S$ and $S \in \Separators,$ let $\RangeSep{S}{X} \coloneqq (V(D) \setminus \Reach{D \setminus S}(X) \rightarrow S \cup \Reach{D \setminus S}(X)).$
    We simplify $\RangeSep{S}{\{v\}}$ to $\RangeSep{S}{v}.$
\end{definition} 

\begin{observation}\ifshort{[\appsymb]}{}\else{}\fi
    \label{sepPrf}
    For all $X \subseteq V(D)$ and $S \in \Separators,$ it holds that $\RangeSep{S}{X} \in \Separations_D^k.$
\end{observation}
\appendixproof{sepPrf}{
\begin{proof}
The object $\RangeSep{S}{X}$ fulfills $V(D) = V_{A}(\RangeSep{S}{X}) \cup V_{A}(\RangeSep{S}{X}).$

In the following we calculate the top side, separator and bottom side of $\RangeSep{S}{X}.$
\begin{align*}
    V_{\top}(\RangeSep{S}{X}) &{}= (V(D) \setminus \Reach{D \setminus S}(X)) \setminus (S \cup \Reach{D \setminus S}(X)) = (V(D) \setminus (S \cup \Reach{D \setminus S}(X)) \\
    V_{S}(\RangeSep{S}{X}) &{}= (V(D) \setminus \Reach{D \setminus S}(X)) \cap (S \cup \Reach{D \setminus S}(X)) = S \\
    V_{\bot}(\RangeSep{S}{X}) &{}= (S \cup \Reach{D \setminus S}(X)) \setminus (V(D) \setminus \Reach{D \setminus S}(X)) = \Reach{D \setminus S}(X)
\end{align*}
By definition of $\Reach{D \setminus S}$ there does not exist an edge $(v,w) \in E(D \setminus S)$ with $v \in \Reach{D \setminus S}(X)$ and $w \in V(D \setminus X) \setminus \Reach{D \setminus S}(X) = (V(D) \setminus (X \cup \Reach{D \setminus S}(X)) = V_{\top}(\RangeSep{S}{X}),$ as this edge would create a path from $X$ to $w$ in $D \setminus S.$ Thus $\RangeSep{S}{X} \in \Separations_D.$ Since $S \in \Separators,$ it follows that $\mathscr{S} \in \Separations_D^k.$
\end{proof}}

\begin{lemma}\ifshort{[\appsymb]}{}\else{}\fi
For all $X, Y \subseteq V(D),$ it holds that $\RangeSep{S}{X} \wedge \RangeSep{S}{Y} = \RangeSep{S}{X \cup Y}.$
\label{uncUp}
\end{lemma}
\appendixproof{uncUp}{
\begin{proof}
By considering the uncrossing, we get
\begin{align*}
    & \RangeSep{S}{X} \wedge \RangeSep{S}{Y} \\
    ={}& (V(D) \setminus \Reach{D \setminus S}(X) \rightarrow S \cup \Reach{D \setminus S}(X)) \wedge (V(D) \setminus \Reach{D \setminus S}(Y) \rightarrow S \cup \Reach{D \setminus S}(Y)) \\
    ={}& (V(D) \setminus \Reach{D \setminus S}(X \cup Y) \rightarrow S \cup \Reach{D \setminus S}(X \cup Y)) 
    = \RangeSep{S}{X \cup Y}.\qedhere
\end{align*}
\end{proof}}

\subsection{\texorpdfstring{$\Separations$}{Separation}-DAGs and DAG decompositions}
\label{sec:S-DAGs}

The notion of $\Separations$-DAGs was introduced in~\cite{dissMeike}.
Let $D$ be a digraph.
An\define{$\Separations$-DAG for $D$} is a tuple $(T, \sigma),$ where $T$ is a DAG of maximum out-degree at most two and $\sigma : E(T) \rightarrow \Separations$ is a function such that for all $t_1, t_2, t_3 \in V(T)$ with $(t_1, t_2), (t_2, t_3) \in E(T ),$ we have $\sigma((t_1, t_2)) \leq \sigma((t_2, t_3)).$
This property is called the \define{consistency} of $\Separations$-DAGs. 
Furthermore, for an $\Separations$-DAG and $t \in V(T),$ we define $\topS(t)$ as the\define{top separation of $t$}, where $\topS(t) \coloneqq (\emptyset \rightarrow V(D))$ if $t$ is a source in $T$ and otherwise $\topS(t) \coloneqq \bigvee_{c \in \inN{T}{t}} \sigma((t,c)).$
Using the top separations, we define $\botS(t)$ as the\define{bottom separation of $t$} where $\botS(t) \coloneqq (V(D) \rightarrow \emptyset)$ if $t$ is a sink in $T$ and otherwise $\botS(t) \coloneqq \bigwedge_{c \in \outN{T}{t}} \topS(c).$
Finally, we define $\bag(t)$ as the\define{bag of $t$} as $\bag(t) \coloneqq V_A(\botS(t)) \cap V_B(\topS(t))$ and we say $(T, \sigma)$ has width $k$ where $k$ is the maximum size of the bags of vertices in $T.$
We also say a digraph $D$ has $\Separations$-DAG-width $k,$ if $k$ is the the smallest number such that there exists a $\Separations$-DAGs of width $k$ for $D.$
\footnote{Note that this is equivalent to the parameter $\nu$-DAG-width from~\cite{dissMeike}.}

Since~\cref{sec:ParGames} is heavily based on the algorithm for parity games on graphs of bounded DAG-width in~\cite{berwanger2006dag}, we also make use of the notion of DAG decompositions introduced by~\cite{obdrzalek2006dag,berwanger2006dag}.
A DAG decomposition of a digraph $D$ is a tuple $(T, (X_t)_{t \in V(T)})$ where $T$ is a DAG and $(X_t)_{t \in V(T)}$ is a family of subsets of $V(D)$ such that
\begin{enamerate}{D}{item:D:last}{2em}
	\item \label{item:D:containsAll}
		$\bigvee_{t \in V(T)} X_t = V(D).$
	\item \label{item:D:laminar}
		$X_t \cap X_{t''} \subseteq X_{t'}$ for all vertices $t,t',t'' \in V(T)$ with $t \preccurlyeq_T t' \preccurlyeq_T t''.$
	\item \label{item:D:edgesGuard}
		$X_t \cap X_{t'}$ guards $X_{\succcurlyeq t'} \setminus X_t$ for all edges $(t, t') \in E(T),$ where $X_{\succcurlyeq t'}$ stands for $\bigvee_{t' \preccurlyeq_T t''} X_{t''}.$ For any source $t,$ $X_{\succcurlyeq t}$ is guarded by $\emptyset.$
	\label{item:D:last}
\end{enamerate}

For a vertex $t \in V(T),$ we also define $V_{\succcurlyeq t} \coloneqq X_{\succcurlyeq t} \setminus X_{t}$ as the vertices under $t$ in $D.$
We call $X_t$ the bag of $t$ and we define the width of a DAG decomposition as the size of the largest bag of a vertex in $T.$
A digraph $D$ has DAG-width $k$ if $k$ is the smallest number such that $D$ has a DAG decomposition of width $k.$
We also use the notion of nice DAG decompositions.
A DAG decomposition $(T, (X_t)_{t \in V(T)})$ of a digraph $D$ is nice if

\begin{enamerate}{N}{item:N:last}{2em}
	\item \label{item:N:uniqueSource}
	       $T$ has a unique source.
    \item \label{item:N:twoOut}
		Every $t \in V(T)$ has at most two successors.
	\item \label{item:N:bothSuccSameBag}
		For $t_0, t_1, t_2 \in V(T),$ if $t_1,$ $t_2$ are two successors of $t_0,$ then $X_{t_0} = X_{t_1} = X_{t_2}.$
	\item \label{item:N:oneSuccOneDiff}
		For $t_0, t_1 \in V(T),$ if $t_1$ is the unique successor of $t_0,$ then $|t_{d_0} \triangle X_{t_1}| = 1.$
	\label{item:N:last}
\end{enamerate}


\subsection{Parity Games}
\label{sec:parity_games_def}
\appendixsection{sec:parity_games_def}

A\define{parity game} is a tuple $(V , V_0, E, \Omega)$ where $(V, E)$ is a digraph, $V_0 \subseteq V$ and $\Omega : V \rightarrow \omega$ is a function assigning a priority to each vertex.
It represents a game played on $V$ by two players\define{Even} and\define{Odd} where the active player (Even, if the current vertex is in $V_0,$ and Odd otherwise) chooses the next move.
A\define{play} of a parity game is an infinite sequence $\pi = (v_i \mid i \in \omega)$ such that $(v_i, v_{i+1}) \in E$ for all $i.$
A\define{partial play} $\pi'$ is a finite or infinite segment of a play in a parity game.
We say a play $\pi$ is\define{winning for Even} if $\lim \inf_{i\rightarrow \infty} \Omega(v_i)$ is even and $\pi$ is\define{winning for Odd} otherwise.
\iflong{That is to say, $\pi$ is winning for Even if the smallest priority that occurs infinitely often in the sequence $(\Omega(v_i))_{i \in \omega}$ is even and it is winning for Odd otherwise.}{}\else{}\fi
A\define{strategy} is a map $f: V^{<\omega} \rightarrow V$ such that, for every sequence $(v_0, \dots , v_i) \in V^{<\omega},$ we have $(v_i , f(v_0, \dots , v_i)) \in E.$
A play $\pi = (v_i \mid i \in \omega)$ is\define{consistent} with Even playing $f$ if $v_i \in V_0$ implies $v_{i+1} = f(v_0, \dots , v_i)$ for $v_i \in \pi.$
Similarly, $\pi$ is consistent with Odd playing $f$ if $v_i \notin V_0$ implies $v_{i+1} = f(v_0, \dots , v_i)$ for $v_i \in \pi.$
A strategy $f$ is\define{winning for Even} from a vertex $v$ if every play beginning at $v$ that is consistent with Even playing $f$ is winning for Even.
\toappendix{A strategy is\define{memoryless} if whenever $u_0, \dots , u_i$ and $v_0, \dots , v_j$ are two sequences in $V^{<\omega}$ with $u_i = v_j$ and $u_i \in V_0,$ then $f(u_0, \dots , u_i ) = f(v_0, \dots , v_j).$
This notion is relevant for parity games since it has been shown that if there is a winning strategy for a player, there is also always a memoryless winning strategy for that player~\cite{emerson1991tree} meaning this strongly restricts the number of strategies to consider.
Furthermore, it is also known that parity games are determined, i.e.~for any game and starting position, either Even or Odd has a winning strategy.}

Parity games find application in different problems in automated verification specifically because they are equivalent to the model-checking of the modal $\mu$-calculus.
The complexity status of parity games is an often considered problem and it has been shown that parity games are solvable in quasi-polynomial time~\cite{calude2017deciding}.
Whether they are solvable in polynomial time remains an open question with the famous complexity conjecture $\textbf{FPT} \neq \textbf{W}[1]$ backing up the claim that they are not solvable in polynomial time.

\section{How to obtain an \texorpdfstring{$\Separations$}{S}-DAG-decomposition}
\label{sec:ConstructSDAG}
\appendixsection{sec:ConstructSDAG}

In this section we outline an algorithm to compute an $\Separations$-DAG of minimum width for a given digraph $D = (V,E).$
To this end we introduce a special variant of the cops-and-robber game.
We then show that the cops have a winning strategy in the game of size $k$ if and only if there exists an $\Separations$-DAG of $D$ with width $k$ and that a winning strategy for the game can be computed in $\mathcal{O}(n^{f(k)})$ time.
For certain steps of the proof it suffices to consider the set of separations $\{\RangeSep{X}{r} \mid \text{$X$ is a cop position, $r$ is a robber position}\}$.
This is crucial for computation time and space as there are only $n^{k+1}$ possible combinations of cop and robber positions whereas the size of the set of all possible separations of order $k$ is $\binom{n}{k} \cdot 2^{n-k} \notin \mathcal{O}(n^{f(k)}).$
We create this game purely for its algorithmic properties and connection to $\Separations$-DAG-width, thus the definition is rather technical. 

\begin{definition}[$\Separations$-DAG cops-and-robber game]
    \label{def:Scr_game}
    Given a digraph $D$ the\define{$\Separations$-DAG cops-and-robber game of size $k$} is a game between the cop and the robber player.
    There are three different types of game states for this game. The set $\CIGS$ of\define{cop initiative game states} is defined as $\CIGS \coloneqq \Separators \times V(D).$ The set $\RIGS_1$ of\define{single separator robber initiative game states} is defined as $\RIGS_1 \coloneqq V(D) \times (\Separators \times [\Separators]^{=1})$ whereas the set $\RIGS_2$ of\define{double separator robber initiative game states} is defined as $\RIGS_2 \coloneqq V(D) \times (\Separators \times [\Separators]^{=2}).$ Together these two sets form the set $\RIGS \coloneqq \RIGS_1 \cup \RIGS_2$ of\define{robber initiative game states}. This results in the set of all game states $\mathcal{G} = \CIGS \cup \RIGS = \CIGS \cup \RIGS_1 \cup \RIGS_2.$ For each of these game states $(r, (X, M)) \in \RIGS$ and $\CIGSex{X}{r} \in \CIGS,$ we call $r$ the\define{robber position} and $X$ the\define{cop position} and, for robber initiative game states, we call $M$ the set of\define{potential cop moves}.
    
    The game is played as follows:
    \begin{figure}[!t]
        \centering
        \begin{tikzpicture}[scale=0.5]
            \def\firstcircle{(-3,0) circle (2.5cm)}
            \def\secondcircle{(-3,-3.5) circle (2.5cm)}
            \def\cutcircle{(-3,0) circle (2.45cm) (-3,-3.5) circle (2.45cm)}
      
            \draw[color=darkviolet] \firstcircle node {$X$};
            \draw[color=darkblue] \secondcircle node {$X'$};
            \begin{scope}
                \clip \firstcircle; 
                \clip \secondcircle;
                \draw[color=black] \cutcircle (-3,-1.75) node {$X_{\bot}$};
            \end{scope}
            \draw (-7,-1.75) node {\ref{item:Mov:C1}};
            \def\Xbox{(7,-3.5) -- (4,-0.5) -- (7,2.5) -- (10,-0.5) -- (7,-3.5)}
            \def\STwo{(5,-1.5) -- (6,-0.5) -- (9.5,-4) -- (8.5,-5) -- (5,-1.5)}
            \def\SOne{(9,-1.5) -- (8,-0.5) -- (4.5,-4) -- (5.5,-5)-- (9,-1.5)}
            \def\XBot{(7,-3.45) -- (8.95,-1.5) (8.95,-1.5) -- (8,-0.55) -- (7,-1.55)  (7,-1.55) -- (6,-0.55) -- (5.05,-1.5) (5.05,-1.5) -- (7,-3.45)}
            \draw[rounded corners=8pt, color=darkviolet] \Xbox;
            \draw[color=darkviolet] (7,-0.5) node {$X$}; 
            \draw[rounded corners=8pt, color=red] \STwo (8.25,-3.75) node {$S_2$};
            \draw[rounded corners=8pt, color=blue] \SOne (5.75,-3.75) node {$S_1$};
            \draw[rounded corners=7.5pt, color=black] \XBot (7,-2.25) node {$X_{\bot}$};
            \draw (3,-1.75) node {\ref{item:Mov:C2}};
        \end{tikzpicture}
        \caption{This figure illustrates the different types of robber initiative game states. 
        For a single separator robber initiative game state the cop player chooses vertices for the cops to move to, such that all newly chosen cops are in the former range of the robber~\cref{item:C1:progressInReach}.
        To leave the robber no way to escape, the remaining cops in $X_{\bot}$ restrict that range~\cref{item:C1:botBlocks}.
        For a double separator robber initiative game state the cop player chooses two different sets $S_{1}$ and $S_{2}$ of vertices for the cops to move to.
        These sets represent two separations whose lower uncrossing has the separator $X_{\bot} \subseteq X$.
        \Cref{item:C2:containIntersection} ensures that $X_{\bot}$ is indeed the separator of this uncrossing.
        \Cref{item:C2:realProgress,item:C2:separatorPartBelow} yield that for every vertex the robber can move to, one of the two separations restricts the robbers range.
        By~\cref{item:C2:botBlocks} the uncrossing restricts the range of the robber ensuring that the robber does not escape.}
        \label{fig:RobberInitiativeGameStates}
    \end{figure}
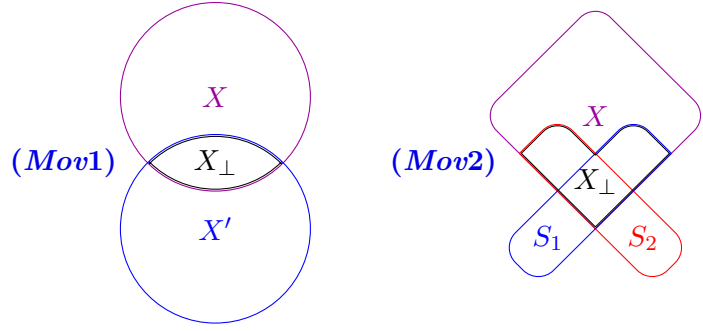
    
    \begin{itemize}
        \item As first move, the robber player chooses $r_0 \in V(D),$ yielding the game state $\CIGSex{\emptyset}{r_0} \in \CIGS.$
        \item From the game state $\CIGSex{X}{r} \in \CIGS$ the cop player can choose to move to game states in $\RIGS_1$ and $\RIGS_2.$ These possibilities are illustrated in \cref{fig:RobberInitiativeGameStates}.
    \end{itemize}
    
    \begin{enamerate}{$\textit{Mov}$}{item:Mov:last}{3em}
        \item\label{item:Mov:C1} To move to a game state in $\RIGS_1$ from a game state $\CIGSex{X}{r} \in \CIGS$ the cop player chooses a set $X' \in \Separators$ and we define $X_{\bot} \coloneqq X' \cap X$ such that the following properties hold:
            \begin{enameratesmall}{$\mathcal{R}_1$}{item:C1:last}{3em}
            	\item \label{item:C1:progressInReach}
            		$X' \setminus X \subseteq \Reach{D \setminus X}(r).$
            	\item \label{item:C1:botBlocks}
            		$\Reach{D \setminus X} (r) = \Reach{D \setminus (X_{\bot})} (r).$
            	\item \label{item:C1:realProgress}
            		$X' \neq X.$
            	\label{item:C1:last}
            \end{enameratesmall}
            This results in the game state $\RIGSOne{r}{X_{\bot}}{X'} \in \RIGS_1.$
        \item\label{item:Mov:C2} To move to a game state in $\RIGS_2$ from the game state $\CIGSex{X}{r} \in \CIGS$ the cop player chooses sets $S_{1}, S_{2} \in \Separators$ and we define $X_{\bot} \coloneqq X \cap (S_{1} \cup S_{2})$ such that the following properties hold:
            \begin{enameratesmall}{$\mathcal{R}_2$}{item:C2:separatorPartBelow}{3em}
            	\item \label{item:C2:containIntersection}
            		$S_{1} \cap S_{2} \subseteq X_{\bot}.$
            	\item \label{item:C2:botBlocks}
            		$\Reach{D \setminus X} (r) = \Reach{D \setminus X_{\bot}} (r).$
            	\item \label{item:C2:realProgress}
            		$\Reach{D \setminus X}(r) \supset \Reach{D \setminus S_{1}} (w)$ or $\Reach{D \setminus X} (r) \supset \Reach{D \setminus S_{2}} (w)$ for all $w \in \Reach{D \setminus X_{\bot}} (r) \setminus (S_{1} \cup S_{2}).$
                \item \label{item:C2:separatorPartBelow}
            	    $\Reach{D \setminus X} (r) \supset \Reach{D \setminus S_{3-i}} (w)$ for all $i \in \{1,2\}$ and $w \in S_{i} \setminus X_{\bot}.$
            	\label{item:C2:last}
        \end{enameratesmall}
        This results in the game state $\RIGSTwo{r}{X_{\bot}}{S_{1}}{S_{2}} \in \RIGS_2.$
    \item\label{item:Mov:R2} From the game state $\RIGSTwo{r}{X_{\bot}}{S_{1}}{S_{2}} \in \RIGS_2$ the robber player chooses $r' \in \Reach{D \setminus X_{\bot}} (r)$ and the cop player chooses $j \in \{1,2\}$ such that $r' \notin S_{j}$ and $\Reach{D \setminus X_{\bot}} (r) \supset \Reach{D \setminus (S_{j})} (r')$ resulting in the game state $\CIGSex{S_{j}}{r'} \in \CIGS.$
        Note that by construction such $r'$ and $j$ always exist.
    \item\label{item:Mov:R1} From the game state $\RIGSOne{r}{X_{\bot}}{X'}$ the robber player chooses $r' \in \Reach{D \setminus X_{\bot}} (r) \setminus X'$ resulting in the game state $\CIGSex{X'}{r'}.$
        \label{item:Mov:last}
    \end{enamerate}
    
    \begin{itemize}
        \item A \textit{play} in the game is a maximal sequence $\pi \coloneqq (\emptyset, r_0), (r_0, (\emptyset, \dots )), \dots$ of game states given by the rules above. Such a sequence is always finite as the graphs we consider are finite and for all finite graphs the number of positions in a play is bounded (see \cref{finitePlay}).
        \item A play $\pi$ is\define{winning for the cop player}, if and only if the last game state is a robber initiative game state.
            A play $\pi$ is \textit{winning for the robber player} if, and only if, it is not winning for the cop player.
        \item A strategy for the cop player in this game is a tuple $(f,g)$ of partial functions $f$ and $g$ with $f: \CIGS \to \RIGS$ and $g: \RIGS_2 \times V(D) \to \Separators$ such that, for every game state $p \in \Def(f),$ the game state $f(p) \in \RIGS$ is a move from $p$ as defined in \cref{item:Mov:C1,item:Mov:C2} and, for every game state $p = \RIGSTwo{r}{X}{S_1}{S_2} \in \RIGS_2$ and $v \in \Reach{D \setminus X}(r)$ with $(p,v) \in \Def(g),$ moving from $p$ to $(g((p,v)),v)$ fulfills \cref{item:Mov:R2}.
        \item A play $\pi$ is\define{consistent with a strategy $(f,g)$ for the cop player}, write $(f,g) \leadsto \pi,$ if, for every game state $p \in \Def(f)$ in $\pi,$ the following game state $p'$ equals $f(p)$ and, for every game state $p \in \RIGS_2$ and its following game state $\CIGSex{X}{r'} \in \CIGS,$ if $(p,r') \in \Def(g),$ then $X = g((p,r')).$
            Additionally a game state $p$ is\define{consistent with $(f,g)$}, write $(f,g) \leadsto p,$ if there exists a play consistent with $(f,g)$ which contains $p.$
            A strategy for the cop player is winning if every play consistent with it is winning for the cop player.
    \end{itemize}
\end{definition}

We call a strategy $(f, g)$ for the cop player in the $\Separations$-DAG cops-and-robber game\define{complete}, if $p \in \Def(f)$ for all $p \in \CIGS$ with $(f, g) \leadsto p$ and $(p,v) \in \Def(g)$ for all $p = \RIGSTwo{r}{X}{S_1}{S_2} \in \RIGS_2$ and $v \in \Reach{D \setminus X}(r).$

\begin{lemma}\ifshort{[\appsymb]}{}\else{}\fi
    \label{strictnessStrat}
    If there exists a winning strategy $(f_0, g_0)$ for the cop player, there also exists a complete winning strategy $(f', g')$ for the cop player.
\end{lemma}
\appendixproof{strictnessStrat}{
\begin{proof}
    We find $f'$ by modifying $f_0.$ For every game state $p$ with $(f, g) \leadsto p$ and $p \notin \Def(f),$ we identify $p' \in \RIGS$ such that the cop player can move to $p'$ from $p$ and define $f'(p) = p'.$ Such a game state $p'$ exists for every such $p$ since no play $\pi$ with $(f, g) \leadsto \pi$ ends on a game state in $\CIGS$ and thus for every game state $p$ with $(f, g) \leadsto p$ there exists at least one game state to move to. We find $g'$ by modifying $g_0.$
    $(f',g)$ is a winning strategy since every play consistent with $(f',g)$ is also consistent with $(f,g).$ For all $p = \RIGSTwo{r}{X}{S_1}{S_2} \in \RIGS_2$ and $v \in \Reach{D \setminus X}(r),$ we define $g'((p,v)) = S_1,$ if $(S_1,v)$ fulfills \cref{item:Mov:R2}, otherwise we define $g'((p,v)) = S_2.$ By \cref{item:C2:separatorPartBelow} and \cref{item:C2:realProgress} one of $(S_1,v)$ and $(S_2,v)$ has to fulfill \cref{item:Mov:R2}. The newly created strategy $(f',g')$ is complete by construction and it is winning since $(f_0,g_0)$ is winning and the plays consistent with $(f',g')$ are a subset of the plays consistent with $(f_0,g_0)$
\end{proof}
}

\begin{lemma}\ifshort{[\appsymb]}{}\else{}\fi
    For every digraph $D,$ if $V(D)$ is finite, every play in the $\Separations$-DAG cops and robber game for $D$ is finite. 
\label{finitePlay}
\end{lemma}
\appendixproof{finitePlay}{
\begin{proof}
    Notice that every play in the $\Separations$-DAG cops and robber game by definition alternates between cop initiative game states and robber initiative game states. \\
    Let $D$ be a digraph such that $V(D)$ is finite.
    Let $\pi$ be a play in an $\Separations$-DAG cops and robber game for $D$ and let $p = \RIGSSpe{r}{X_{\bot}}{P} \in \RIGS$ be a robber initiative game state in $\pi,$ let $p_{+1} = \CIGSex{S}{w} \in \CIGS$ with $S \in P$ be the game state directly after $p$ in $\pi$ and let $p_{-1} = \CIGSex{X}{r} \in \CIGS$ be the game state directly before $p$ in $\pi.$
    Since $p_{+1}$ is a legal move from $p$ it must hold that $\Reach{D \setminus S}(w) \subset \Reach{D \setminus X}(r).$
    Thus, for any pair of consecutive cop initiative game states, the reach of the robber is strictly decreasing.  
    Let $(\emptyset, s) \in \CIGS$ be the starting game state of a play. By definition $|\Reach{D}(s)| \leq |V(D)|.$
    As the reach of the robber decreases with every cop initiative game state the number of cop initiative game states in a play is at most $|V(D)|+1$ leading to a total number of game states of at most $2 |V(D)|+2.$
    \label{proFinitePlay}
\end{proof}}

\newcommand{\CopChoiceSet}[1]{\mathcal{CC}(#1)}
\begin{lemma}[\appsymb]
    \label{theoStr->Dec}
    If the cop player has a winning strategy in the $\Separations$-DAG cops-and-robber game of size $k$ on $D$, there exists an $\Separations$-DAG of width $k$ for $D$.
\end{lemma}
\begin{proof}[Proofsketch]
    We show that from a winning strategy $(f,g)$ for the cop player in the $\Separations$-DAG cops-and-robber game of size $k$ on $D$ we can build an $\Separations$-DAG of width $k$ for $D$. 

Let $(f, g)$ be complete by~\cref{strictnessStrat}. 
We construct an $\Separations$-DAG which represents every game state $\CIGSex{X}{r} \in \CIGS$ for $X \in \Separators$, $r \in V(D)$ and $(f,g) \leadsto (X,r)$, with the separation $\RangeSep{X}{r}$.
The rest of the construction is built to connect these separations into an $\Separations$-DAG in accordance with $(f,g)$.

Let us first define a useful notion for this construction. This notion is the cop choice set $\CopChoiceSet{X,S,r}$. This set contains the vertices $w$ in $\Reach{D \setminus X} (r) \setminus S$ for which there exists a game state $(r, (X, M))$ with $S \in M$ where the cop player chooses to pursue the robber by going to $S$ if the robber moves to $w$. Explicitly, if $(f,g) \leadsto \RIGSOne{r}{X}{S}$, we define $\CopChoiceSet{X,S,r}$ to be $\Reach{D \setminus X} (r) \setminus S$ and otherwise we define $\CopChoiceSet{X,S,r}$ to contain all $w \in \Reach{D \setminus X} (r) \setminus S$
for which there exists $M \in [\Separators]^{=2}$ with $(f,g) \leadsto (r, (X, M))$ and $g((r, (X, M)), w) = S$.
We say\define{for $w \in \CopChoiceSet{X,S,r}$ the cop player chooses $S$ under $X$ with respect to $r$}.

We now build an $\Separations$-DAG $(T,\sigma)$ of width $k$ for $D$ from $(f,g)$ using the following steps:
\begin{enumerate}[label=\textbf{(\arabic*)}]
    \item\label{step:1} For every game state $\CIGSex{X}{r} \in  \CIGS$  with $(f,g) \leadsto \CIGSex{X}{r}$ we create a vertex $t_{X,r}$.
    \item\label{step:2} For every game state $p = \RIGSTwo{r}{X_{\bot}}{S_{1}}{S_{2}} \in \RIGS_2$ add two vertices $t_{(X_{\bot}, S_{1}), r}$ and $t_{(X_{\bot}, S_{2}), r}$ to $T$.
        Different game states may create vertices with the same name in this step in which case they refer to the same vertex.
        For every play $\pi$ with $(f,g) \leadsto \pi$ and $p \in \pi$, let $\CIGSex{X}{r}$ be the game state directly before $p$ in $\pi$.
        We add edges from $t_{X,r}$ to $t_{(X_{\bot}, S_{1}), r}$ and $t_{(X_{\bot}, S_{2}), r}$ and define $\sigma((t_{X,r},t_{(X_{\bot}, S_{j}), r})) = \RangeSep{S_j}{\CopChoiceSet{X_{\bot},S_j,r}}$ for $j \in \{1,2\}$. 
    \item\label{step:3} For every game state $p=\RIGSOne{r}{X_{\bot}}{X'} \in \RIGS_1$, add two vertices $c_{(X_{\bot}, X'), r}$ and $t_{(X_{\bot}, X'), r}$ to $T$ and, for every play $\pi$ with $(f,g) \leadsto \pi$ and $p \in \pi$, let $\CIGSex{X}{r}$ be the game state directly before $p$ in $\pi$.
        We add the edges from $(t_{X,r},c_{(X_{\bot}, X'), r})$ and $(c_{(X_{\bot}, X'), r},t_{(X_{\bot}, X'), r})$ and set $\sigma((t_{X,r},c_{(X_{\bot}, X'), r})) \coloneqq \RangeSep{X_{\bot}}{r}$ and $\sigma((c_{(X_{\bot}, X'), r},t_{(X_{\bot}, X'), r})) \coloneqq \RangeSep{X'}{\CopChoiceSet{X_{\bot},X',r}}$.
    \item\label{step:4} For every vertex $t \in V(T)$ of the form $t_{(X, S), r}$ use it as the root for a directed binary tree $T_{X,S,r}$ of height $\lceil \log_2(|\CopChoiceSet{X,S,r}|) \rceil$ with $\{t_{S,w} | w \in \CopChoiceSet{X,S,r}\}$ as the leaves.
        Note that in case there is only one leaf the root $t_{(X, S), r}$ is identified with the leaf.
        We call $S$ the\define{separator of $T_{X,S,r}$} as every edge in $T_{X,S,r}$ has $S$ as its separator.

        For a digraph $T'$ and $v \in V(T')$, let $\LeafReach{T'}(v) \coloneqq \{w \in V(D) \mid \exists S' $ with $t_{S',w} \in \Reach{T'}(v)\}$.
        Note that $S'$ here is always the separator of $T'$.
        We use $\LeafReach{T'}(t)$ to define $\sigma$ for the edges in $T_{X,S,r}$.     
        For every edge $e=(t_1,t_2)$ in $T_{X,S,r}$ we choose $\sigma(e) \coloneqq \RangeSep{S}{\LeafReach{T_{X,S,r}}(t_2)}$.
    \item\label{step:5} Finally, we create a source vertex $t_0$ for $T$.
        We use $t_0$ as the root for a directed binary tree $T_0$ of height $\lceil \log_2(|V(D)|) \rceil$ with $\{t_{\emptyset,w} | w \in V(D)|\}$ as leaves.
        Note that this makes $t_0$ the only source vertex in $T$.
        Here we call $\emptyset$ the separator of $T_0$.
        For every edge $e=(t_1,t_2)$ in $T_0$ we choose $\sigma (e) \coloneqq \RangeSep{\emptyset}{\LeafReach{T_0}(t_2)}$.
\end{enumerate}
Using \cref{uncUp}, we obtain that the resulting tuple $(T,\sigma)$ is an $\Separations$-DAG of width $k$ for $D$. 
\end{proof}
\appendixproof{theoStr->Dec}{
\begin{proof}
We show that from a winning strategy $(f,g)$ for the cop player in the $\Separations$-DAG cops-and-robber game of size $k$ on $D$ we can build an $\Separations$-DAG of width $k$ for $D$. 

Let $(f, g)$ be complete by~\cref{strictnessStrat}. 
We construct an $\Separations$-DAG which represents every game state $\CIGSex{X}{r} \in \CIGS$ for $X \in \Separators$, $r \in V(D)$ and $(f,g) \leadsto (X,r)$, with the separation $\RangeSep{X}{r}$.
The rest of the construction is built to connect these separations into an $\Separations$-DAG in accordance with $(f,g)$.

Let us first define a useful notion for this construction. This notion is the cop choice set $\CopChoiceSet{X,S,r}$. This set contains the vertices $w$ in $\Reach{D \setminus X} (r) \setminus S$ for which there exists a game state $(r, (X, M))$ with $S \in M$ where the cop player chooses to pursue the robber by going to $S$ if the robber moves to $w$. Explicitly, if $(f,g) \leadsto \RIGSOne{r}{X}{S}$, we define $\CopChoiceSet{X,S,r}$ to be $\Reach{D \setminus X} (r) \setminus S$ and otherwise we define $\CopChoiceSet{X,S,r}$ to contain all $w \in \Reach{D \setminus X} (r) \setminus S$
for which there exists $M \in [\Separators]^{=2}$ with $(f,g) \leadsto (r, (X, M))$ and $g((r, (X, M)), w) = S$.
We say\define{for $w \in \CopChoiceSet{X,S,r}$ the cop player chooses $S$ under $X$ with respect to $r$}.

We now build an $\Separations$-DAG $(T,\sigma)$ of width $k$ for $D$ from $(f,g)$ with the following steps:
\begin{enumerate}[label=\textbf{(\arabic*)}]
    \item For every game state $\CIGSex{X}{r} \in  \CIGS$  with $(f,g) \leadsto \CIGSex{X}{r}$ we create a vertex $t_{X,r}$.
    \item For every game state $p = \RIGSTwo{r}{X_{\bot}}{S_{1}}{S_{2}} \in \RIGS_2$ add two vertices $t_{(X_{\bot}, S_{1}), r}$ and $t_{(X_{\bot}, S_{2}), r}$ to $T$. Note that different game states may create vertices with the same name in this step in which case they refer to the same vertex.
        For every play $\pi$ with $(f,g) \leadsto \pi$ and $p \in \pi$, let $\CIGSex{X}{r}$ be the game state directly before $p$ in $\pi$.
        We add edges from $t_{X,r}$ to $t_{(X_{\bot}, S_{1}), r}$ and $t_{(X_{\bot}, S_{2}), r}$. We define $\sigma((t_{X,r},t_{(X_{\bot}, S_{j}), r})) = \RangeSep{S_j}{\CopChoiceSet{X_{\bot},S_j,r}}$ for $j \in \{1,2\}$. 
    \item For every game state $p=\RIGSOne{r}{X_{\bot}}{X'} \in \RIGS_1$, add two vertices $c_{(X_{\bot}, X'), r}$ and $t_{(X_{\bot}, X'), r}$ to $T$ and, for every play $\pi$ with $(f,g) \leadsto \pi$ and $p \in \pi$, let $\CIGSex{X}{r}$ be the game state directly before $p$ in $\pi$.
        We add edges from $t_{X,r}$ to $c_{(X_{\bot}, X'), r}$ and from $c_{(X_{\bot}, X'), r}$ to $t_{(X_{\bot}, X'), r}$.
        We define $\sigma((t_{X,r},c_{(X_{\bot}, X'), r})) = \RangeSep{X_{\bot}}{r}$ and $\sigma((c_{(X_{\bot}, X'), r},t_{(X_{\bot}, X'), r})) = \RangeSep{X'}{\CopChoiceSet{X_{\bot},X',r}}$.
    \item For every vertex $t \in V(T)$ of the form $t_{(X, S), r}$ use it as the root for a directed binary tree $T_{X,S,r}$ of height $\lceil \log_2(|\CopChoiceSet{X,S,r}|) \rceil$ with $\{t_{S,w} | w \in \CopChoiceSet{X,S,r}\}$ as the leaves.
        Note that in case there is only one leaf the root $t_{(X, S), r}$ is identified with the leaf.
        We call $S$ the\define{separator of $T_{X,S,r}$} as every edge in $T_{X,S,r}$ has $S$ as its separator.
    
        To define $\sigma$ for the edges in this tree using a set $\LeafReach{T'}(t)$ containing the vertices of $D$ represented by the leaves of the tree $T'$ which are reachable from a vertex $t \in V(T')$:
        For a digraph $T'$ and $v \in V(T')$, let $\LeafReach{T'}(v) \coloneqq \{w \in V(D) \mid \exists S' $ with $t_{S',w} \in \Reach{T'}(v)\}$.
        Note that $S'$ here is always the separator of $T'$.
        For every edge $e=(t_1,t_2)$ in $T_{X,S,r}$ we choose $\sigma(e) \coloneqq \RangeSep{S}{\LeafReach{T_{X,S,r}}(t_2)}$.
    \item Finally we create a source vertex $t_0$ for $T$.
        We use $t_0$ as the root for a directed binary tree $T_0$ of height $\lceil \log_2(|V(D)|) \rceil$ with $\{t_{\emptyset,w} | w \in V(D)|\}$ as leaves.
        Note that this makes $t_0$ the only source vertex in $T$.
        Here we call $\emptyset$ the separator of $T_0$.
        For every edge $e=(t_1,t_2)$ in $T_0$ we choose $\sigma (e) \coloneqq \RangeSep{\emptyset}{\LeafReach{T_0}(t_2)}$.
\end{enumerate}
We now prove that the resulting tuple $(T,\sigma)$ is an $\Separations$-DAG of width $k$ for $D$. 

By~\cref{sepPrf} the co-domain of $\sigma$ are indeed separations.

For every vertex there is at most one step which creates outgoing edges for it.
The construction in the steps \cref{step:2,step:3} create exactly two children and one child per vertex respectively since $(f,g)$ is complete and winning for the cop player. 
Thus, the digraph $T$ has maximum out-degree two as the constructions in the steps~\cref{step:4,step:5} are binary trees. 

Further note that $T$ does not contain a cycle as a cycle in $T$ would have to contain some vertex $t_{X,r}$ resulting in an infinite play $p$ with $(f,g) \leadsto p$ contradicting~\cref{finitePlay}.

We go through the different types of vertices created in the construction and consider their top separation. 
\begin{itemize}
    \item Let $t_{X,r}$ be a vertex created in the step~\cref{step:1}. 
    
        The incoming edges for $t_{X,r}$ are normally constructed in the step~\cref{step:4} or~\cref{step:5} where $t_{X,r}$ is a leaf in some tree(s) $T'$.
        We know that $\LeafReach{T'}(t_{X,r}) = \{r\}$ since $t_{X,r}$ is always a leaf in $T'$.
        Thus, the separation on those edges is $\RangeSep{X}{r}$ which is in both cases fully determined by $X$ and $r$ and thus all incoming edges have the same separation assigned to them, meaning that $\RangeSep{X}{r} = \topS(t_{X,r})$.
        
        Note that in the special case, where $t_{X,r}$ is the only leaf of $T'$ and thus it is also a vertex of the form $t_{(X^{-}_{\bot}, X), r^{-}}$, it has incoming edges created in step~\cref{step:2} or~\cref{step:3}.
        In this case $\{r\} = \CopChoiceSet{X^{-}_{\bot},X,r^{-}} = \LeafReach{T'}(t_{X,r})$ still resulting in the same separation assigned to the incoming edge.
        \begin{align*}
            \RangeSep{X}{r}&{}= (V(D) \setminus \Reach{D \setminus X}(r) \rightarrow X \cup \Reach{D \setminus X}(r)) \\
            &{}= (V(D) \setminus \{r\} \rightarrow X \cup \{r\}) \\
            &{}= (V(D) \setminus \Reach{D \setminus X}(\CopChoiceSet{X^{-}_{\bot},X,r}) \rightarrow X \cup \Reach{D \setminus X}(\CopChoiceSet{X^{-}_{\bot},X,r^{-}})) \\
            &{}= \RangeSep{X}{\CopChoiceSet{X^{-}_{\bot},X,r^{-}}}
        \end{align*}
    \item Let $t_{(X_{\bot}, S), r}$ be a vertex created in step~\cref{step:2} or~\cref{step:3}.
    
        The only incoming edges for $t_{(X_{\bot}, S), r}$ are created in step~\cref{step:2} or~\cref{step:3}.
        Similar to the first case, for all incoming edges $e$ of $t_{(X_{\bot}, S), r}$, the value $\sigma(e) = \RangeSep{S}{\CopChoiceSet{X_{\bot},S,r}}$ is only dependent on $X_{\bot}, S$ and $r$.
        Thus, this separation is identical for all incoming edges meaning that $\RangeSep{S}{\CopChoiceSet{X_{\bot},S,r}} = \topS(t_{(X_{\bot}, S), r})$. 
    \item Let $c_{(X_{\bot}, S), r}$ be a vertex created in step~\cref{step:3}
    The only incoming edges for $c_{(X_{\bot}, S), r}$ are created in step~\cref{step:3} with the separation $\RangeSep{X_\bot}{r}$ making this separation its top separation.
    \item Let $t$ be an inner vertex of a tree $T'$ in the step~\cref{step:4} or~\cref{step:5} of the construction such that $S'$ is the separator of $T'$.
        Then, since $T'$ is a tree, $t$ only has a single incoming edge whose assigned separation $\RangeSep{S'}{\LeafReach{T'}(t)}$ is $\topS(t)$.
    \item For the source vertex $t_0$ the set $\topS(t_0)$ is $(\emptyset, V(D))$ by definition.
\end{itemize}
From this we can defer that, for every vertex $t \in V(T)$, the separation, which $\sigma$ assigns to its incoming edges, is identical to $\topS(t)$.
Thus, for all children $t'$ of $t$ in $T$, we know that $\botS(t) \leq \sigma((t,t'))$ because $\botS(t) = \bigwedge_{c \in \outN{T}{t}}\sigma((t,c)) \leq \sigma((t,t'))$.
We now prove for every vertex $t \in V(T)$, that its bag size is at most $k$ and that $\topS(t) \leq \botS(t)$. This then shows that $\sigma$ is consistent as, $\sigma((t_{pre},t))= \topS(t) \leq \botS(t) \leq \sigma((t, t_{post}))$ for each pair of an incoming edge $(t_{pre},t) \in E(T)$ and an outgoing edge $(t, t_{post}) \in E(T)$ of $t$.
\begin{itemize}
    \item Let $t_{X,r}$ be a vertex created in the step~\cref{step:1}
    
        Since $(f,g)$ is winning for the cop player and complete, there is exactly one game state $p$ after the game state $\CIGSex{X}{r}$.
        The outgoing edges from $t_{X,r}$ are created in the step~\cref{step:2} or~\cref{step:3} and both of these cases result in the same bottom separation.
    
        We first consider the case that the outgoing edges are created in step~\cref{step:3}.
        In this case we know that $p$ is of the form $p=\RIGSOne{r}{X_{\bot}}{X'} \in \RIGS_1$ thus the only outgoing edge is to the vertex $c_{(X_{\bot}, X'), r}$. Its top separation is $\RangeSep{X_{\bot}}{r}$ which also makes it $\botS(t_{X,r})$.
        Secondly, we consider the case that the outgoing edges are created in step~\cref{step:2}.
        In this case we know that $p \in \RIGS_2$.
        Let $p= \RIGSTwo{r}{X_{\bot}}{S_1}{S_2}$ then the outgoing edges from $t_{X,r}$ are to the vertices $t_{(X_{\bot}, S_{1}), r}$ and $t_{(X_{\bot}, S_{2}), r}$. Their top separations are $\RangeSep{S_1}{\CopChoiceSet{X_{\bot},S_1,r}}$ and $
        \RangeSep{S_2}{\CopChoiceSet{X_{\bot},S_2,r}}$.
        Resulting in: 
        \begin{align*}
            \botS(t_{X,r}) ={}& \RangeSep{S_1}{\CopChoiceSet{X_{\bot},S_1,r}} \wedge \RangeSep{S_2}{\CopChoiceSet{X_{\bot},S_2,r}}\\
            ={}& \left( V(D) \setminus \Reach{D \setminus S_1}(\CopChoiceSet{X_{\bot},S_1,r}) \rightarrow S_1 \cup \Reach{D \setminus S_1}(\CopChoiceSet{X_{\bot},S_1,r})\right) \\
            &\wedge \left(V(D) \setminus \Reach{D \setminus S_2}(\CopChoiceSet{X_{\bot},S_2,r}) \rightarrow S_2 \cup \Reach{D \setminus S_2}(\CopChoiceSet{X_{\bot},S_2,r})\right) \\
            ={}& \Big(V(D) \setminus \left(\Reach{D \setminus S_1}(\CopChoiceSet{X_{\bot},S_1,r}) \cup \Reach{D \setminus S_2}(\CopChoiceSet{X_{\bot},S_2,r})\right) \\
            &\rightarrow S_1 \cup S_2 \cup \Reach{D \setminus S_1}(\CopChoiceSet{X_{\bot},S_1,r}) \cup \Reach{D \setminus S_2}(\CopChoiceSet{X_{\bot},S_2,r})\Big) \\
            \stackrel{(3)}{=} {}& (V(D) \setminus \Reach{D \setminus X_\bot}(r) \rightarrow S_1 \cup S_2 \cup \Reach{D \setminus X_\bot}(r)) \\
            \stackrel{(2)}{=} {}& (V(D) \setminus \Reach{D \setminus X_\bot}(r) \rightarrow X_{\bot} \cup \Reach{D \setminus X_\bot}(r)) \\
            ={}& \RangeSep{X_{\bot}}{r}.
        \end{align*}
        The equality $(2)$ follows from the fact that all vertices in $S_1$ and $S_2$ are either in $X_{\bot}$ or in $\Reach{D \setminus X_\bot}(r)$ as, for all $w \in S_i \setminus X_{\bot}$, by definition of the game $\Reach{D \setminus X_{\bot}} (r) \supset \Reach{D \setminus (S_{i,1-j})} (w)$ and thus $w \in \Reach{D \setminus X_{\bot}} (r)$.
        The equality $(3)$ follows from the fact that $\CopChoiceSet{X_{\bot},S_1,r} \cup \CopChoiceSet{X_{\bot},S_2,r} = \Reach{D \setminus X_\bot}(r)$. 
        
        It holds that $\topS(t_{X,r}) \leq \botS(t_{X,r})$ since $V(D) \setminus \Reach{D \setminus X}(r) \subseteq V(D) \setminus \Reach{D \setminus X_\bot}(r)$ and $X_{\bot} \cup \Reach{D \setminus X_\bot}(r) \subseteq X \cup \Reach{D \setminus X}(r)$.
        
        It also follows: 
        \begin{align*}
            \bag(t_{X_{\bot},r}) &{}= (X_{\bot} \cup \Reach{D \setminus X_{\bot}}(r)) \cap (V(D) \setminus \Reach{D \setminus X_{\bot}}(r)) \\ 
            &{}= (X \cup \Reach{D \setminus X_{\bot}}(r)) \cap (V(D) \setminus \Reach{D \setminus X_\bot}(r)) \\
            &{}= X.
        \end{align*}
        For the set $X$, by definition $|X| \leq k$.
    \item Let $t_{(X_{\bot}, S), r}$ be a vertex created in step~\cref{step:2} or~\cref{step:3} which is the root of the tree $T_{X_{\bot},S,r}$.
    
    We first consider the case that $t_{(X_{\bot}, S), r}$ is a sink. In this case $\botS(t_{(X_{\bot}, S), r}) = (V(D),\emptyset)$ and $\CopChoiceSet{X_{\bot}, S, r}$ is empty.
    Checking for laminarity is not necessary since there are no outgoing edges.
    Regarding $\bag(t_{(X_{\bot}, S), r})$ it follows:
    \begin{align*}
        \bag(t_{X_{\bot},r}) &{}= (S \cup \CopChoiceSet{X_{\bot},S,r}) \cap V(D) \\
        &{}= S \cap V(D) \\
        &{}= S.
    \end{align*}
    The set $S$ has size smaller or equal to $k$ by definition.
    
    Second, we consider the case that $t_{(X_{\bot}, S), r}$ does have outgoing edges. Notice that the case in which the tree $T_{X_{\bot},S,r}$ only has one leaf has already been dealt with in the first case, so we assume the tree has at least two leaves and thus $t_{(X_{\bot}, S), r}$ has out-degree two.
    These two outgoing edges are created in the construction of the tree in the step~\cref{step:4} to the two children $t_1$ and $t_2$ in the tree $T_{X_{\bot},S,r}$.
    By definition $\LeafReach{T_{X_{\bot},S,r}}(t_{(X_{\bot}, S), r}) = \LeafReach{T_{X_{\bot},S,r}}(t_1) \cup \LeafReach{T_{X_{\bot},S,r}}(t_2)$ as $t_{(X_{\bot}, S), r}$ is not a leaf and thus all leafs reachable from $t_{(X_{\bot}, S), r}$ are also reachable from one of its children.
    
    To make the following equations more easily readable we define: $T' = T_{X_{\bot},S,r}$.
    This results in: 
    \begin{align*}
        \botS(t_{(X_{\bot}, S), r}) 
        &{}= \RangeSep{S}{\LeafReach{T'}(t_1)} \wedge \RangeSep{S}{\LeafReach{T'}(t_2)} \text{\hspace*{\fill}} \\
        &{}\stackrel{\text{\ref{uncUp}}}{=} \RangeSep{S}{\LeafReach{T'}(t_1) \cup \LeafReach{T'}(t_2)} \\
        &{}= \RangeSep{S}{\LeafReach{T'}(t_{(X_{\bot}, S), r})} \\
        &{}= \RangeSep{S}{\CopChoiceSet{X_{\bot},S,r}}. \\
    \end{align*}
    This separation is equal to $\topS(t_{(X_{\bot}, S), r})$ for which the separator is $S$ with $|S| \leq k$. 
    \item Let $c_{(X_{\bot}, X'), r}$ be a vertex created in step~\cref{step:3}.
    Its only outgoing edge is to the vertex $t_{(X_{\bot}, X'), r}$.
    This vertex has the top separation $\RangeSep{S}{\CopChoiceSet{X_{\bot},X',r}}$ making it the bottom separation for $c_{(X_{\bot}, X'), r}$. We know that $\topS(c_{(X_{\bot}, X'), r}) \leq \botS(c_{(X_{\bot}, X'), r})$ since $V(D) \setminus \Reach{D \setminus X_{\bot}}(r) \subseteq V(D) \setminus \Reach{D \setminus X'}(\CopChoiceSet{X_{\bot},X',r})$ and $X_{\bot} \cup \Reach{D \setminus X_{\bot}}(r) \supseteq X' \cup \Reach{D \setminus X'}(\CopChoiceSet{X_{\bot},X',r})$.
    
    Additionally, recalling that in this case $\CopChoiceSet{X_{\bot},X',r} = \Reach{D \setminus X_{\bot}}(r) \setminus X'$ and $X' \setminus X_{\bot} \subseteq \Reach{D \setminus X_{\bot}}(r)$, it follows:
    \begin{align*}
        \bag(c_{(X_{\bot}, X'), r}) &{}= (X_{\bot} \cup \Reach{D \setminus X_{\bot}}(r)) \cap (V(D) \setminus \Reach{D \setminus X'}(\CopChoiceSet{X_{\bot},X',r})) \\
        &{}= (X_{\bot} \cup \Reach{D \setminus X_{\bot}}(r)) \cap (V(D) \setminus \CopChoiceSet{X_{\bot},X',r}) \\
        &{}= (X_{\bot} \cup (X' \cap \Reach{D \setminus X_{\bot}}(r)) \cup (\Reach{D \setminus X_{\bot}}(r) \setminus X')) \\
        &{}\cap (V(D) \setminus \CopChoiceSet{X_{\bot},X',r}) \\
        &{}= (X' \cup \CopChoiceSet{X_{\bot},X',r}) \cap (V(D) \setminus \CopChoiceSet{X_{\bot},X',r}) \\
        &{}= X'.
    \end{align*}
    By definition $X' \in \Separators$.
    \item Let $t$ be an inner vertex of a tree $T'$ in step\cref{step:4} or~\cref{step:5} of the construction such that $S'$ is the separator of $T'$.
    
    Let $t_1$ and $t_2$ be the two children of $t$ in the tree $T'$. This case is analogous to the second case. Note that $\LeafReach{T'}(t) = \LeafReach{T'}(t_1) \cup \LeafReach{T'}(t_2)$.
    It follows: 
    \begin{align*}
        \botS(t) &{}=\RangeSep{S'}{\LeafReach{T'}(t_1)} \wedge \RangeSep{S'}{\LeafReach{T'}(t_2)} \\
        &{}\stackrel{\text{\ref{uncUp}}}{=} \RangeSep{S'}{\LeafReach{T'}(\{t_1,t_2\})} \\
        &{}= \RangeSep{S'}{\LeafReach{T'}(t)}.
    \end{align*}
    So here the bottom separation is again identical to the top separation with the separator being $S'$ with $|S'| \leq k$. Also the top and bottom separations are equal and thus the separation assigned to the incoming edge is laminar to the separation assigned to the outgoing edges. 
    \item For the source vertex $t_0$, consider that the robber can choose all vertices in the first turn. Since a vertex can be reached from itself, it follows that $\LeafReach{T_0}(t_0) = V(D)$.
    
    Let $t_1, t_2 \in V(D)$ be the children of $t_0$ in $T_0$.
    Then the bottom set can be calculated like it is calculated in the third case resulting in
    \begin{align*}
        \botS(t_0) &{}= \RangeSep{\emptyset}{\LeafReach{T_0}(t_0)} \\
        &{}= \RangeSep{\emptyset}{V(D) \setminus X_0}.
    \end{align*}
    Thus the separator is $\emptyset \in \Separators$.
    In the special case in which there is only a single vertex $v$ in $V(D)$ the vertex $t_0$ is also a leaf of $T_0$ called $t_{\emptyset,v}$ and its only child $c_{(X_{0,\bot}, X_1), v}$ is created in step~\cref{step:3} with the separation $\RangeSep{\emptyset}{v}$ resulting in the separator $\emptyset$ as well.
\end{itemize}
In conclusion, since $T$ is a DAG of maximum out-degree two, $\sigma$ is consistent and $|\bag(t)| \leq k$ for all $t \in V(D)$, the tuple $(T, \sigma)$ is an $\Separations$-DAG for $D$ of width at most $k$.
\end{proof}}


\begin{definition}
A\define{nice $\Separations$-DAG $(T, \sigma)$ of width $k$} for $D$ is an $\Separations$-DAG of width $k$ for $D$ with the following properties: 
\begin{enamerate}{N}{item:nice:childrenUseful}{2em}
	\item \label{item:nice:uniqueSource}
		$T$ has a unique source $s$ with $\bag(s)=\emptyset$.
	\item \label{item:nice:normedTop}
		$\sigma((v,w)) = \topS(w)$ for every edge $(v,w) \in E(T)$.
    \item \label{item:nice:onlyOne}
		$|\bag(t) \triangle \bag(c)| \leq 1$ for all $t \in V(T)$ with $|\outN{D}{t}| = 1$ and $c \in \outN{D}{t}$.
	\item \label{item:nice:splitNoChange}
		$\topS(c) = \botS(c)$ for all $t \in V(T)$ with $|\outN{D}{t}| = 2$ and $c \in \outN{D}{t}$.
	\item \label{item:nice:childrenUseful}
		$\sigma((v,w_1))$ and $\sigma((v,w_2))$ cross for all $(v,w_1),(v,w_2) \in E(T)$ with $w_1 \neq w_2$.
	\label{item:nice:last}
\end{enamerate}
\label{niceDecDef}
\end{definition}

\begin{theorem}[\appsymb]
    If there exists an $\Separations$-DAG $(T, \sigma)$ of width $k$ for $D$, there exists a nice $\Separations$-DAG $(T', \sigma')$ of width $k$ for $D$ with $|T'| \in O(|T| \cdot |V(D)|)$.
    \label{niceExists}
\end{theorem}

\appendixproof{niceExists}{
\begin{proof}
    Let $(T, \sigma)$ be an $\Separations$-DAG of width $k$ for $D$. Based on $(T, \sigma)$ we construct a nice $\Separations$-DAG $(T', \sigma')$ of width $k$ for $D$.

    If $T$ has multiple sources, let $R$ be the set of sources of $T$. We construct a binary tree with $R$ as the leaves, For every edge $e$ in the tree we choose $\sigma'(e)$ to be the minimum separation and add it to $T'$. If $T$ only has a single source, we add a single vertex and add an edge to the only source with the minimum separation assigned to it. In both cases, this results in $T'$ fulfilling \ref{item:nice:uniqueSource}. The result of the construction remains an $\Separations$-DAG of order $k$ as the bags of the vertices in $R$ remain unchanged, the bags of all new vertices are empty and the minimum separation is smaller or equal to all separations.
    
    For all $(u,v) \in E(T)$, we leverage the construction from \cite[Lemma 3.4.15]{dissMeike} and modify $\sigma$ to $\sigma'$ with $\sigma'((u,v)) = \topS(v)$, fulfilling \ref{item:nice:normedTop}. This construction results in an $\Separations$-DAG as the top separations are unchanged and $\sigma'$ fulfills consistency as in all $\Separations$-DAGs $\topS(v) \leq \sigma((v,w)) \leq \topS(w)$.

    For all $t \in V(T)$ and $c \in \outN{D}{t}$ with $|\outN{D}{t}| = 1$ and $l = |\bag(t) \triangle \bag(c)| > 1$, let $U = \{u_1, \dots, u_m\} \coloneqq \bag(t) \setminus \bag(c)$ and let $W = \{w_1, \dots, w_n\} \coloneqq \bag(c) \setminus \bag(t)$ such that $m+n=l$.
    It follows that $\sigma((t,c)) = (V_A(\botS(c)) \setminus W \rightarrow V_B(\topS{t}) \setminus U)$.
    We replace $t$ and $c$ with $l+1$ different vertices $c_i$, for $i \in [1,l+1]$, such that the incoming edges to $t$ go to $c_1$ instead and the outgoing edges from $c$ start in $c_{l+1}$ instead. We also introduce edges from $c_i$ to $c_{i+1}$ for $i \in [1,l]$. For $i \in [1,m]$, we set $\sigma((c_{i},c_{i+1}))= (V_A(\botS(c)) \setminus W \rightarrow V_B(\topS{t}) \setminus \{u_{1},\dots,u_i\})$ and for $i \in [1,n]$, we set $\sigma((c_{m+i},c_{m+i+1}))= (V_A(\botS(c)) \setminus \{w_{i+1},\dots,w_{n}\} \rightarrow V_B(\topS{t}) \setminus U)$.
    All the given values for $\sigma$ are separations since they are all variations of $\sigma((t,c))$ with a bigger separator and they are all laminar to each other.
    The resulting bags are all at most of size $k$ since $c_0$ has the same bag size as $t$ and $c_l$ has the same bag size as $c$, while the bag for all other created vertices is a strict subset of one of these two bags.
    Applying this construction for all such pairs of $t$ and $c$ satisfies \ref{item:nice:onlyOne} as the bags for all created vertices differ by exactly one from their in-neighbour and out-neighbour.
    
    For all $t \in V(T)$ and $c \in \outN{D}{t}$ with $|\outN{D}{t}| = 2$ and $\topS(c) \neq \botS(c)$, we split $c$ into $c_{\top}$ and $c_{\bot}$ such that the incoming edges of $c$ go to $c_{\top}$ instead and the outgoing edges from $c$ start in $c_{\bot}$ instead. We introduce the edge $(c_{\top},c_{\bot})$ with $\sigma((c_{\top},c_{\bot}))= \topS(c)$. This construction fulfills \ref{item:nice:splitNoChange}, since $\topS(c_{\top}) = \botS(c_{\top})$ and $c_{\top}$ has out-degree $1$. Also it is still an $\Separations$-DAG of width $k$ as the bag size for $c_{\top}$ and $c_{\bot}$ is smaller or equal to $c$ and the separations assigned to their incident edges are laminar. 

    For all $(v,w_1),(v,w_2) \in E(T)$ with $w_1 \neq w_2$ and $\sigma((v,w_1)) \leq \sigma((v,w_2))$ we delete the edge $(v,w_2)$. The bottom separation of $v$ remains unchanged since $\botS(v)= \topS(w_1) \wedge \topS(w_2) = \sigma((v,w_1)) \wedge \sigma((v,w_2)) = \sigma((v,w_2))$. In case $w_2$ has another incoming edge, its top separation also remains unchanged by \ref{item:nice:normedTop}. In case $w_2$ has no other incoming edge we delete it and iteratively delete all sources generated in this way. This process is shown exemplary in \cref{fig:partDAGdel}. This deletion of vertices still results in an $\Separations$-DAG as there are still vertices remaining and for all remaining vertices other then $v$ their outgoing edges are unchanged and their top separation is unchanged as their remaining incoming edges still have their original top separation assigned to them by \ref{item:nice:normedTop}. Thus the generated graph is an $\Separations$-DAG fulfilling \ref{item:nice:childrenUseful}. 

    The steps ensuring \ref{item:nice:splitNoChange} and \ref{item:nice:uniqueSource} each create at parts linearly to the number of vertices in $T$, while the construction for \ref{item:nice:onlyOne} creates less then $2 \cdot V(D)$ new parts for the $\Separations$-DAG. Thus the total size of the resulting construction $T'$ is in $O(|T| \cdot |V(D)|)$.
    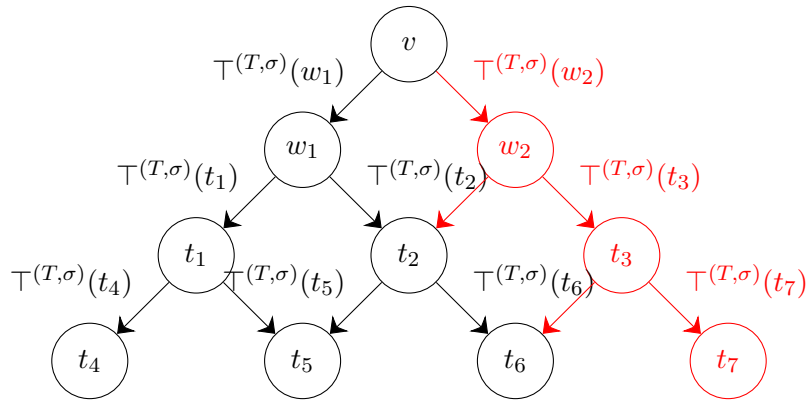
\begin{figure}[!ht]
        \centering
        \begin{tikzpicture}[scale=0.7,>={Stealth[width=3mm,length=2mm]}]
        \node[draw, circle,minimum size=1cm] (v) at (0,0) {$v$};
        \node[draw, circle,minimum size=1cm] (w1) at (-2,-2) {$w_1$};
        \node[draw, circle,minimum size=1cm] (t1) at (-4,-4) {$t_1$};
        \node[draw, circle,minimum size=1cm] (t2) at (0,-4) {$t_2$};
        \node[draw, circle,minimum size=1cm] (t4) at (-6,-6) {$t_4$};
        \node[draw, circle,minimum size=1cm] (t5) at (-2,-6) {$t_5$};
        \node[draw, circle,minimum size=1cm] (t6) at (2,-6) {$t_6$};
        \node[draw, circle,minimum size=1cm, color=red] (w2) at (2,-2) {$w_2$};
        \node[draw, circle,minimum size=1cm, color=red] (t3) at (4,-4) {$t_3$};
        \node[draw, circle,minimum size=1cm, color=red] (t7) at (6,-6) {$t_7$};
        \path[->] (v) edge node[above left] {$\topS(w_1)$} (w1)
        (w1) edge node[above left] {$\topS(t_1)$} (t1)
        (t1) edge node[above left] {$\topS(t_4)$} (t4)
        (t2) edge node[above left] {$\topS(t_5)$} (t5)
        (w1) edge node[above right] {$\topS(t_2)$} (t2)
        (t1) edge (t5)
        (t2) edge node[above right] {$\topS(t_6)$} (t6);
        \path[->,color=red] (v) edge node[above right] {$\topS(w_2)$} (w2)
        (w2) edge (t2)
        (w2) edge node[above right] {$\topS(t_3)$} (t3)
        (t3) edge (t6)
        (t3) edge node[above right] {$\topS(t_7)$} (t7);
        \end{tikzpicture}
        \caption{This figure illustrates the process of deleting an edge $(v,w_2)$ and all newly resulting sources ($w_2,t_6,t_7$) as an example for the process in the proof for \cref{niceExists}. The part of the graph that is deleted is in red. The bottom separation for $v$ is unchanged as $\sigma((v,w_1)) \leq \sigma((v,w_2))$ and thus $\botS(v) = \sigma((v,w_1)) \wedge \sigma((v,w_2)) = \sigma((v,w_1))$. For all remaining vertices the top separations are unchanged as their incoming edges still have their original top separations assigned to them.}
        \label{fig:partDAGdel}
    \end{figure}
\end{proof}}

\begin{theorem}[\appsymb]
    \label{theoDec->Str}
    If there exists an $\Separations$-DAG of width $k$ for $D$ then the cop player has a winning strategy in the $\Separations$-DAG cops-and-robber game of size $k$ on $D$.
\end{theorem}
\begin{proof}[Proofsketch]
    We assume there exists an $\Separations$-DAG of width $k$ for $D$.
    By \cref{niceExists}, there also exists a nice $\Separations$-DAG of width $k$ for $D$.
    We show that from a nice $\Separations$-DAG of width $k$ for $D$ we can build a winning strategy $(f, g)$ for the cop player in the $\Separations$-DAG cops-and-robber game of size $k$ on $D$.
    
    Let $(T, \sigma)$ be a nice $\Separations$-DAG of width $k$ for $D$.
    We construct a complete strategy $(f, g)$ by iteratively designating a vertex we call a\define{forward-facing vertex in $T$} for every cop initiative game states consistent with our strategy. 
    We use these forward facing vertices to identify a following robber initiative game state.
    We also introduce a vertex we call a\define{cop choice vertex for $v$ under $p$ in $T$} for each robber initiative game state $p = \RIGSSpe{r}{X}{P}$ and vertex $v \in V(D)$, where $(f,g) \leadsto p$ and $v \in \Reach{D \setminus X}(r)$.
    The cop choice vertex is used to find a forward facing vertex for $p' = \CIGSex{S}{v}$ where $S \in P$.
    This process is illustrated in~\cref{fig:FFAndCC}.
    
    A forward-facing vertex $x$ for a game state $\CIGSex{X}{r} \in \CIGS$ has to fulfill the following properties: 
    \begin{enamerate}{FF}{item:FF:children_intersect}{3em}
    	\item \label{item:FF:X_is_above}
    		$X \subseteq V_A(\botS(x))$. -- ``covering''
    	\item \label{item:FF:x_blocks}
    		$\Reach{D \setminus (X \cap \bag(x))}(r) = \Reach{D \setminus X}(r)$. -- ``blocking''
    	\item \label{item:FF:Reach_r_under}
    		$\Reach{D \setminus X}(r) \subseteq V_{\bot}(\botS(x))$. -- ``controlled''
    	\item \label{item:FF:children_intersect}
    		$\Reach{D \setminus X}(r) \cap \bag(c) \neq \emptyset$ for all $c \in \outN{T}{x}$. -- ``forward facing''
    	\label{item:FF:last}
    \end{enamerate}
    
    A cop choice vertex $t$ for $v$ under $p = \RIGSSpe{r}{X_{\bot}}{P} \in \RIGS$ has to fulfill the property that there exists an $S \in P$ with $\Reach{D \setminus S}(v) \subset \Reach{D \setminus X}(r)$, $\Reach{D \setminus S}(v) \subseteq V_{\bot}(\botS(t))$ and $S \subseteq V_A(\botS(t))$.
    
    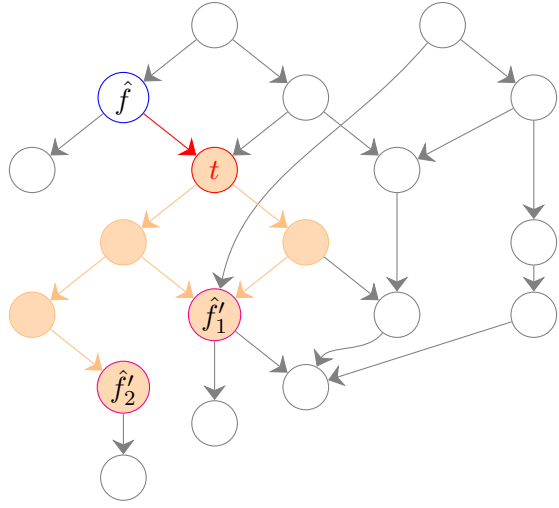
\begin{figure}[!t]
    \centering
    \begin{tikzpicture}[scale=0.6,yscale=0.8,>={Stealth[width=3mm,length=2mm]}]
        \tikzset{
    mystyle/.style={
      circle,
      inner sep=0pt,
      text width=5mm,
      align=center,
      minimum size=0.6cm
      }
    }
    \colorlet{tDAGcolour}{orange!30!white}
    \colorlet{tDAGcolourframe}{orange!50!white}
        \node[draw=gray, mystyle] (w1) at (-2,-2) { };
        \node[draw=gray, mystyle] (w2) at (3,-2) { };
        \node[draw=gray, mystyle] (t7) at (5,-4) { };
        \node[draw, mystyle,draw=blue] (t1) at (-4,-4) {$\hat{f}$};
        \node[draw=gray,mystyle] (t2) at (0,-4) { };
        \node[draw=gray,mystyle] (t4) at (-6,-6) { };
        \node[draw, mystyle, color=red, fill=tDAGcolour] (t) at (-2,-6) {$t$};
        \node[draw=gray, mystyle] (t6) at (2,-6) { };
        \node[draw, mystyle, draw=tDAGcolourframe, fill=tDAGcolour] (g1) at (-4,-8) { };
        \node[draw, mystyle, draw=tDAGcolourframe, fill=tDAGcolour] (g2) at (0,-8) { };
        \node[draw, mystyle, draw=tDAGcolourframe, fill=tDAGcolour] (g3) at (-6,-10) { };
        \node[draw, mystyle, draw=magenta, fill=tDAGcolour] (c1) at (-2,-10) {$\hat{f}'_1$};
        \node[draw, mystyle, draw=magenta, fill=tDAGcolour] (c2) at (-4,-12) {$\hat{f}'_2$};
        \node[draw=gray,mystyle] (o1) at (0,-12) { };
        \node[draw=gray,mystyle] (o3) at (-2,-13) { };
        \node[draw=gray,mystyle] (o2) at (-4,-14.5) { };
        \node[draw=gray,mystyle] (t8) at (2,-10) { };
        \node[draw=gray,mystyle] (t9) at (5,-8) { };
        \node[draw=gray,mystyle] (t10) at (5,-10) { };
        \path[->, color=gray]
        (c1) edge (o1)
        (c2) edge (o2)
        (c1) edge (o3)
        (g2) edge (t8)
        (t6) edge (t8)
        (w2) edge[out=235,in=80] (c1)
        (w1) edge (t1)
        (w2) edge (t7)
        (t7) edge (t6)
        (t1) edge (t4)
        (t2) edge (t)
        (w1) edge (t2)
        (t2) edge (t6)
        (t8) edge[out=235,in=65] (o1)
        (t7) edge (t9)
        (t9) edge (t10)
        (t10) edge (o1);
        \path[->, color=red]
        (t1) edge (t);
        \path[->, color=tDAGcolourframe]
        (t) edge (g1)
        (t) edge (g2)
        (g1) edge (g3)
        (g1) edge (c1)
        (g2) edge (c1)
        (g3) edge (c2);
    \end{tikzpicture}
        \caption{This figure illustrates the process of finding a forward facing vertex for a game state $p=\CIGSex{X}{r}$ from the cop choice vertex $t$ for $r$ under $p_{-1}$.
        The red circled cop choice vertex $t$ is a child of the forward facing vertex $\hat{f}$(in blue) of the last cop initiative game state $p_{-2}$.
        In light orange is the DAG $T_{t,r,X}$, which contains all vertices $v$ starting from $t$ which lie above $\Reach{D \setminus X}(r)$ meaning that $\Reach{D \setminus X}(r) \subseteq V_{\bot}(\botS(v))$.
        All sinks of this DAG are marked in pink.
        These sinks $\hat{f}'_1$ and $\hat{f}'_2$ are the possible choices for the forward-facing vertex for $p$.}
        \label{fig:FFAndCC}
    \end{figure}
    
    Given the forward-facing vertex $x$ for a game state $p =
    \CIGSex{X}{r} \in \CIGS$, we define $f(p)$ and $g((p,v))$ for $v \in \Reach{D \setminus X}(r)$ as well as the following cop choice vertices.
    If $x$ has one child $c$, the cop player chooses to move to $X' = \bag(c) \cap (X \cup \Reach{D \setminus X}(r))$ resulting in the value $f(p) = \RIGSOne{r}{X_{\bot}}{X'}$ where $X_{\bot}=X \cap \bag(c)$.
    For $v \in \Reach{D \setminus X}(r) \setminus X'$, we also designate $c$ as the cop choice vertex for $v$ under $f(p)$.  
    If $x$ has two children $c_1, c_2$, we define 
    \begin{align*}
        X_{\bot} ={}& X \cap (\bag(c_1) \cup \bag(c_2)), \quad
        S_1 ={} \bag(c_1) \cap (X \cup \Reach{D \setminus X}(r)), \\
        S_2 ={}& \bag(c_2) \cap (X \cup \Reach{D \setminus X}(r)), \quad
        f(p) ={} \RIGSTwo{r}{X_{\bot}}{S_1}{S_2}, \\
        g(f(p),v) ={}& 
        \begin{cases}
            S_1 & \text{if } v \in V_{\bot}(\topS(c_1)) \cap \Reach{D \setminus X}(r) \\
            S_2 & \text{if } v \in V_{\bot}(\topS(c_2)) \cap \Reach{D \setminus X}(r) \setminus V_{\bot}(\topS(c_1)).
        \end{cases}
    \end{align*}
    For $v \in \Reach{D \setminus X}(r)$, if $v \in V_{\bot}(\topS(c_1))$, we designate $c_1$ as the cop choice vertex for $v$ under $f(p)$, otherwise we designate $c_2$ as the cop choice vertex for $v$ under $f(p)$.
    While $g$ and thus the cop choice vertices depends on which child is designated as $c_1$, this choice can be made arbitrarily as the construction works independently of how that choice is made.
    

    To find the forward facing vertices, we first consider a simple construction.
    For $t \in V(T), r \in V(D)$ and $X \in \Separators$ with $V_{\bot}(\botS(t)) \supseteq \Reach{D \setminus X}(r)$ we define $T_{t,r,X}$ to be the induced subgraph of $T$ on the set $\{w \in \Reach{T}(t) \mid V_{\bot}(\botS(w)) \supseteq \Reach{D \setminus X}(r) \}$.
    This graph is non-empty as it always contains $t$ and it is connected as the bottom separations for any path in $T$ are laminar and thus, for every path from $t$ to another vertex $w$ with $V_{\bot}(\botS(w)) \supseteq \Reach{D \setminus X}(r)$, every vertex $v$ on that path has to fulfill $V_{\bot}(\botS(w)) \supseteq \Reach{D \setminus X}(r)$ as well.  
    We now lead an induction to find a forward facing vertex for every cop initiative game state in a play in which all previous cop initiative game states have forward facing vertices. 
    
    Let $p_s=\CIGSex{\emptyset}{v} \in \CIGS$ be a starting game state. 
    As $V_{\bot}(\botS(s)) = V(D) \supseteq \Reach{D}(v)$, we can construct $T_{s,v,\emptyset}$.
    We choose any sink $x$ of $T_{s,v,\emptyset}$ to be the forward-facing vertex for $\CIGSex{\emptyset}{v}$.
    
    
    Let $p=\CIGSex{X}{v} \in \CIGS$ be a non-starting game state in a play $\pi$ such that $p_{-2} \in \CIGS$ is the last cop initiative game state before $p$ in $\pi$, $p_{-1} = f(p_{-2}) \in \RIGS$ is the robber initiative game state directly before $p$ in $\pi$ and if $p_{-1} \in \RIGS_2$, then $X = g((p_{-1},v))$.
    Let $t$ be the cop choice vertex for $v$ under $p_{-1}$.
    Note that there exists a cop choice vertex for every game state the robber can move to.
    Then, by definition, of the cop choice vertices we know that $V_{\bot}(\botS(t)) \supseteq \Reach{D \setminus X}(v)$.
    So we can construct $T_{t,v,X}$.
    Again we choose any sink $x$ of $T_{t,v,X}$ to be the forward-facing vertex for $\CIGSex{X}{v}$, which satisfies \cref{item:FF:X_is_above,item:FF:x_blocks,item:FF:Reach_r_under,item:FF:children_intersect}.
    
    Since we can assign a forward facing vertex to every cop initiative game state $p$ consistent with $(f, g)$ and define $f(p)$ for all cop initiative game states consistent with $(f, g)$, we obtain that $(f, g)$ is complete and thus there exists a legal move after every cop initiative game state, meaning that no play consistent with $(f, g)$ ends on a cop initiative game state.
    By~\cref{finitePlay} every play is finite, therefore $(f, g)$ is a winning strategy for the cop player.
\end{proof}

\appendixproof{theoDec->Str}{
\begin{proof}
We assume there exists an $\Separations$-DAG of width $k$ for $D$.
By \cref{niceExists}, there also exists a nice $\Separations$-DAG of width $k$ for $D$.
We show that from a nice $\Separations$-DAG of width $k$ for $D$ we can build a winning strategy $(f, g)$ for the cop player in the $\Separations$-DAG cops-and-robber game of size $k$ on $D$.

Let $(T, \sigma)$ be a nice $\Separations$-DAG of width $k$ for $D$.
We construct a complete strategy $(f, g)$ by iteratively designating a\define{forward-facing vertex in $T$} for all cop initiative game states consistent with our strategy. 
We use these forward facing vertices to identify a following robber initiative game state.
We also designate a\define{cop choice vertex for $v$ under $p$ in $T$} for each robber initiative game state $p = \RIGSSpe{r}{X}{P}$ and vertex $v \in V(D)$, where $(f,g) \leadsto p$ and $v \in \Reach{D \setminus X}(r)$.
The cop choice vertex is used to find a forward facing vertex for $p' = \CIGSex{S}{v}$ where $S \in P$.
This process is illustrated in~\cref{fig:FFAndCC}.

A forward-facing vertex $x$ for a game state $\CIGSex{X}{r} \in \CIGS$ has to fulfill the following properties: 
\begin{enamerate}{FF}{item:FF:lastapp}{3em}
	\item 
		$X \subseteq V_A(\botS(x))$. -- ``covering''
	\item 
		$\Reach{D \setminus (X \cap \bag(x))}(r) = \Reach{D \setminus X}(r)$. -- ``blocking''
	\item 
		$\Reach{D \setminus X}(r) \subseteq V_{\bot}(\botS(x))$. -- ``forward facing''
	\item 
		$\Reach{D \setminus X}(r) \cap \bag(c) \neq \emptyset$ for all $c \in \outN{T}{x}$. -- ``outreaching''
	\label{item:FF:lastapp}
\end{enamerate}

A cop choice vertex $t$ for $v$ under $p = \RIGSSpe{r}{X_{\bot}}{P} \in \RIGS$ has to fulfill the property that there exists an $S \in P$ with $\Reach{D \setminus S}(v) \subset \Reach{D \setminus X}(r)$, $\Reach{D \setminus S}(v) \subseteq V_{\bot}(\botS(t))$ and $S \subseteq V_A(\botS(t))$.

\begin{figure}[!ht]
\centering
\begin{tikzpicture}[scale=0.6,>={Stealth[width=3mm,length=2mm]}]
    \node[draw=gray, circle,minimum size=1cm] (w1) at (-2,-2) { };
    \node[draw=gray, circle,minimum size=1cm] (w2) at (3,-2) { };
    \node[draw=gray, circle,minimum size=1cm] (t7) at (5,-4) { };
    \node[draw, circle,minimum size=1cm,draw=blue] (t1) at (-4,-4) {$\hat{f}$};
    \node[draw=gray, circle,minimum size=1cm] (t2) at (0,-4) { };
    \node[draw=gray, circle,minimum size=1cm] (t4) at (-6,-6) { };
    \node[draw, circle,minimum size=1cm, color=red, fill=lightgreen] (t) at (-2,-6) {$t$};
    \node[draw=gray, circle,minimum size=1cm] (t6) at (2,-6) { };
    \node[draw, circle,minimum size=1cm, draw=green, fill=lightgreen] (g1) at (-4,-8) { };
    \node[draw, circle,minimum size=1cm, draw=green, fill=lightgreen] (g2) at (0,-8) { };
    \node[draw, circle,minimum size=1cm, draw=green, fill=lightgreen] (g3) at (-6,-10) { };
    \node[draw, circle,minimum size=1cm, draw=magenta, fill=lightgreen] (c1) at (-2,-10) {$\hat{f}'_1$};
    \node[draw, circle,minimum size=1cm, draw=magenta, fill=lightgreen] (c2) at (-4,-12) {$\hat{f}'_2$};
    \node[draw=gray, circle,minimum size=1cm] (o1) at (0,-12) { };
    \node[draw=gray, circle,minimum size=1cm] (o3) at (-2,-13) { };
    \node[draw=gray, circle,minimum size=1cm] (o2) at (-4,-14.5) { };
    \node[draw=gray, circle,minimum size=1cm] (t8) at (2,-10) { };
    \node[draw=gray, circle,minimum size=1cm] (t9) at (5,-8) { };
    \node[draw=gray, circle,minimum size=1cm] (t10) at (5,-10) { };
    \path[->, color=gray]
    (c1) edge (o1)
    (c2) edge (o2)
    (c1) edge (o3)
    (g2) edge (t8)
    (t6) edge (t8)
    (w2) edge[bend right= 10] (c1)
    (w1) edge (t1)
    (w2) edge (t7)
    (t7) edge (t6)
    (t1) edge (t4)
    (t2) edge (t)
    (w1) edge (t2)
    (t2) edge (t6)
    (t8) edge (o1)
    (t7) edge (t9)
    (t9) edge (t10)
    (t10) edge (o1);
    \path[->, color=red]
    (t1) edge (t);
    \path[->, color=green]
    (t) edge (g1)
    (t) edge (g2)
    (g1) edge (g3)
    (g1) edge (c1)
    (g2) edge (c1)
    (g3) edge (c2);
\end{tikzpicture}
    \caption{This figure illustrates the process of finding a forward facing vertex for a game state $p=\CIGSex{X}{r}$ from the cop choice vertex $t$ for $r$ under $p_{-1}$. The red circled cop choice vertex $t$ is a child of the forward facing vertex $\hat{f}$(in blue) of the last cop initiative game state $p_{-2}$. In green is the DAG $T_{t,r,X}$, which contains all vertices $v$ starting from $t$ which lie above $\Reach{D \setminus X}(r)$ meaning that $\Reach{D \setminus X}(r) \subseteq V_{\bot}(\botS(v))$. All sinks of this DAG are marked in pink. These sinks $\hat{f}'_1$ and $\hat{f}'_2$ are the possible choices for the forward-facing vertex for $p$.}
    \label{fig:FFAndCCapp}
\end{figure}
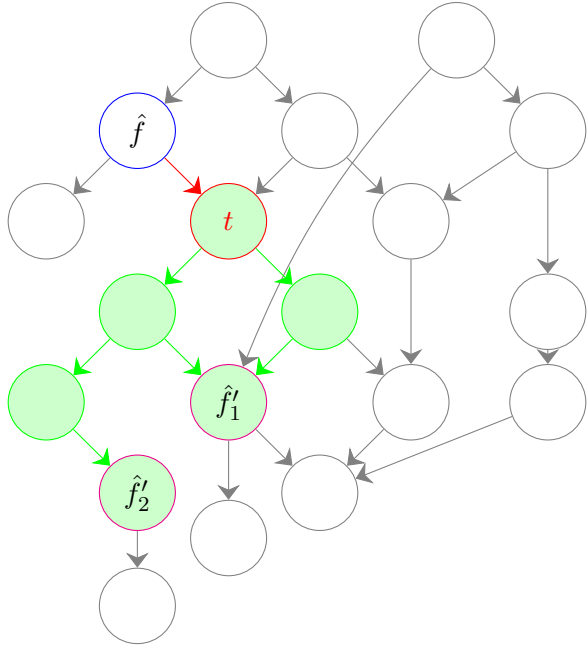

Given the forward-facing vertex $x$ for a game state $p =
\CIGSex{X}{r} \in \CIGS$, we define $f(p)$ and $g((p,v))$ for $v \in \Reach{D \setminus X}(r)$ as well as the following cop choice vertices.
If $x$ has one child $c$, the cop player chooses to move to $X' = \bag(c) \cap (X \cup \Reach{D \setminus X}(r))$ resulting in the value $f(p) = \RIGSOne{r}{X_{\bot}}{X'}$ where $X_{\bot}=X \cap \bag(c)$.
For $v \in \Reach{D \setminus X}(r) \setminus X'$, we also designate $c$ as the cop choice vertex for $v$ under $f(p)$.
The game state $f(p)$ is a legal move from $p$, because
$\bag(c) \cap (X \cup \Reach{D \setminus X}(r)) \subseteq \bag(c)$ and thus it is in $\Separators$. We also calculate $X_{\bot}$ as $X_{\bot} = X' \cap X = (\bag(c) \cap (X \cup \Reach{D \setminus X}(r))) \cap X = \bag(c) \cap X$.
We consider the conditions:
\begin{itemize}
    \item \ref{item:C1:realProgress}: By definition of a forward facing vertex $\Reach{D \setminus X}(r) \cap \bag(c) \neq \emptyset$ and thus 
    \begin{align*}
        X' \setminus X ={}& (\bag(c) \cap (X \cup \Reach{D \setminus X}(r))) \setminus X \\
        ={}& \bag(c) \cap \Reach{D \setminus X}(r) \setminus X \\
        ={}& \bag(c) \cap \Reach{D \setminus X}(r) \\
        \neq& \emptyset.
    \end{align*} 
    \item \ref{item:C1:botBlocks}:
    \begin{align*}
    \Reach{D \setminus X}(r)
    \stackrel{\ref{item:FF:x_blocks}}{=} {}& \Reach{D \setminus (X \cap \bag(x))}(r) \\
    \stackrel{(1)}{=} {}& \Reach{D \setminus (X \cap \bag(x) \cap \bag(c))}(r) \\
    \stackrel{(2)}{=} {}& \Reach{D \setminus (X \cap \bag(c))}(r) \\
    ={}& \Reach{D \setminus (X_{\bot})}(r).
    \end{align*}
    As $c$ is a child of $x$, it follows from \ref{item:FF:X_is_above} that 
    \begin{equation*}
    \bag(c) \cap V_A(\botS(x)) \subseteq \bag(x) \cap V_A(\botS(x)) = \bag(x)
    \end{equation*}
    which shows the correctness of equality $(2)$.
    The equality $(1)$ follows from the fact that there exist no edges from $V_{\bot}(\botS(x))$ to $V_{\top}(\botS(x))$, while $\Reach{D \setminus X}(r) \subseteq V_{\bot}(\botS(x))$ and $X \setminus \bag(c) \subseteq V_{\top}(\topS(c)) = V_{\top}(\botS(x))$.
    \item \ref{item:C1:progressInReach}:
    \begin{align*}
X' \setminus X_{\bot} ={}& (\bag(c) \cap (X \cup \Reach{D \setminus X}(r))) \setminus (X \cap \bag(c)) \\ 
 ={}& (\bag(c) \cap \Reach{D \setminus X}(r)) \setminus X \\ 
    \subseteq {}& \Reach{D \setminus X}(r).
    \end{align*}
\end{itemize}
By definition $X' \subseteq \bag(c) = V_{S}(\botS(c)) \subseteq V_{A}(\botS(c))$.
By \ref{item:FF:children_intersect}, the intersection between $\Reach{D \setminus X}(r)$ and $X'$ is non-empty. Since $\Reach{D \setminus X}(r) \setminus X' \subseteq V_{\bot}(\botS(c))$, all outgoing edges from $\Reach{D \setminus X}(r) \setminus X'$ end in $X' \cup V_{B}(\botS(c)) \subseteq V_{B}(\topS(c))$. Since, by definition all outgoing edges from $\Reach{D \setminus X}(r) \setminus X'$ end in $X' \cup X$, it follows that all outgoing edges from $\Reach{D \setminus X}(r) \setminus X'$ end in $(X' \cup X) \cap V_{B}(\topS(c)) = X'$. Since $v \in \Reach{D \setminus X}(r) \setminus X'$, $\Reach{D \setminus X'}(v) \subset \Reach{D \setminus X}(r)$ and $\Reach{D \setminus X'}(v) \subseteq V_{\bot}(\botS(c))$, we know that $c$ fulfills the condition to be the cop choice vertex for $v$ under $f(p)$.

If $x$ has two children $c_1, c_2$, we define 
\begin{align*}
    X_{\bot} ={}& X \cap (\bag(c_1) \cup \bag(c_2)) \\
    S_1 ={}& \bag(c_1) \cap (X \cup \Reach{D \setminus X}(r)) \\
    S_2 ={}& \bag(c_2) \cap (X \cup \Reach{D \setminus X}(r)) \\
    f(p) ={}& \RIGSTwo{r}{X_{\bot}}{S_1}{S_2} \\
    g(f(p),v) ={}& 
    \begin{cases}
        S_1 & \text{if } v \in V_{\bot}(\topS(c_1)) \cap \Reach{D \setminus X}(r) \\
        S_2 & \text{if } v \in V_{\bot}(\topS(c_2)) \cap \Reach{D \setminus X}(r) \setminus V_{\bot}(\topS(c_1)).
    \end{cases}
\end{align*}
For $v \in \Reach{D \setminus X}(r)$, if $v \in V_{\bot}(\topS(c_1))$, we designate $c_1$ as the cop choice vertex for $v$ under $f(p)$, otherwise we designate $c_2$ as the cop choice vertex for $v$ under $f(p)$.
While the definition of $g$ and as an extension of that the definition of the cop choice vertices depends on which child is designated as $c_1$, this choice can be made arbitrarily as the construction works independently of how that choice is made.
As before we show that $f(p)$ is a legal move from $p$ and that the cop choice vertices are chosen appropriately. We also show that $(g((f(p),v)),v)$ is a legal move from $f(p)$. The sets chosen for $S_{i,1}$ and $S_{i,2}$ are subsets of $\bag(c_1)$ and $\bag(c_2)$ respectively and thus they are in $\Separators$. We can also calculate $X_{\bot}$ as 
\begin{align*}
        & X \cap (\bag(c_1) \cup \bag(c_2)) \\
        ={}& X \cap (X \cap (\bag(c_1) \cup \bag(c_2))) \\
        ={}& X \cap ((X \cap \bag(c_1)) \cup (X \cap \bag(c_2))) \\
        \stackrel{(1)}{=} {}& X \cap ((\Reach{D \setminus X}(r) \cap (\bag(c_2) \cup \bag(c_1))) \cup ((X \cap \bag(c_1)) \cup (X \cap \bag(c_2)))) \\
        ={}&X \cap (((\Reach{D \setminus X}(r) \cup X) \cap \bag(c_1)) \cup ((\Reach{D \setminus X}(r) \cup X) \cap \bag(c_2))) \\
        ={}& X \cap (S_1 \cup S_2) \\
        ={}& X_{\bot}.
\end{align*}
The equality $(1)$ follows from $X \cap (\Reach{D \setminus X}(r) \cap (\bag(c_2) \cup \bag(c_1))) \subseteq X \cap \Reach{D \setminus X}(r) = \emptyset$. 
We consider the conditions:
\begin{itemize}
    \item \ref{item:C2:containIntersection}:
    This follows from the fact that $\botS(x)$ is the uncrossing of $\topS(c_1)$ and $\topS(c_2)$ and thus $V_{S}(\botS(x)) \supseteq V_{S}(\topS(c_1)) \cap V_{S}(\topS(c_1)) = \bag(c_1) \cap \bag(c_1)$. Since $V_{S}(\botS(x)) \cap \Reach{D \setminus X}(r) = \emptyset$ it follows that $\bag(c_1) \cap \bag(c_1) \cap \Reach{D \setminus X}(r) = \emptyset$.
    \begin{align*}
        S_1 \cap S_2 ={}& (\bag(c_1) \cap (X \cup \Reach{D \setminus X}(r))) \cap (\bag(c_2) \cap (X \cup \Reach{D \setminus X}(r))) \\ 
        ={}& \bag(c_1) \cap \bag(c_2) \cap (X \cup \Reach{D \setminus X}(r)) \\
        ={}& \bag(c_1) \cap \bag(c_2) \cap X \\
        \subseteq {}& X \cap (\bag(c_1) \cup \bag(c_2)) \\
        ={}& X_{\bot}
    \end{align*}
    \item \ref{item:C2:botBlocks}:
    Notice that by definition of $\Reach{D \setminus X}(r)$ all outgoing edges from the set go to vertices in $X$. Also note that $\Reach{D \setminus X}(r) \subseteq V_{\bot}(\botS(x))$ while $X \subseteq V_{A}(\botS(x))$ and thus all outgoing edges from $\Reach{D \setminus X}(r)$ go into $V_{S}(\botS(x)) \cap X$.
    \begin{align*}
    \Reach{D \setminus X}(r) ={}& \Reach{D \setminus (V_{S}(\botS(x)) \cap X)}(r) \\
        \stackrel{(1)}{=} {}& \Reach{D \setminus (X \cap (\bag(c_1) \cup \bag(c_2)))}(r) \\
        ={}& \Reach{D \setminus X_{\bot}}(r)
    \end{align*}
    The equality $(1)$ is correct as $\botS(x) = \topS(c_1) \wedge \topS(c_2)$ and it follows that, since $X$ is in $V_{A}(\topS(c_1))$ and $V_{A}(\topS(c_2))$, every vertex in $X$ and $V_{B}(\topS(c_1))$ or $V_{B}(\topS(c_2))$ is in $V_{S}(\botS(x))$.
    \item \ref{item:C2:realProgress} and \ref{item:C2:separatorPartBelow}:
    For $i \in \{1,2\}$ and $w \in S_i \setminus X_{\bot}$, since $X_{\bot} = X \cap (S_1 \cup S_2)$ and $S_i \subseteq X \cup \Reach{D \setminus X}(r)$, we know that $w \in \Reach{D \setminus X}(r)$.
    
    Let $w \in \Reach{D \setminus X}(r)$. 
    \begin{align}
        w \in{}& \Reach{D \setminus X}(r) \nonumber \\
        \subseteq{}& V_{\bot}(\botS(x)) \label{thisIsW} \\
        ={}& V_{\bot}(\botS(c_1)) \cup V_{\bot}(\botS(c_2)) \nonumber
    \end{align}
    Let $j \in \{1,2\}$ be chosen such that $w \in V_{\bot}(\botS(c_j))$. By \ref{item:FF:children_intersect} $\Reach{D \setminus X}(r) \cap \bag(c_j) \neq \emptyset$. It follows that $\Reach{D \setminus X}(r) \cap S_j \neq \emptyset$ and thus $\Reach{D \setminus S_j}(w) \neq \Reach{D \setminus X}(r)$. $\Reach{D \setminus S_j}(w) =  \Reach{D \setminus (V_{S}(\topS(c_j)) \cap X)}(w) \subseteq \Reach{D \setminus X}(r)$ since $X \subset V_A(\topS(c_j))$ and all outgoing edges from $\Reach{D \setminus X}(r)$ in $D$ are to $X$. Since $\Reach{D \setminus S_j}(w) \subseteq \Reach{D \setminus X}(r)$ and $\Reach{D \setminus S_j}(w) \neq \Reach{D \setminus X}(r)$, it follows that $\Reach{D \setminus S_j}(w) \subset \Reach{D \setminus X}(r)$.
\end{itemize}

As all conditions are fulfilled, $f(p)$ is a legal move for the cop player. We proceed by considering $g((f(p),v),v)$. As outlined in \cref{thisIsW}, for all $v \in \Reach{D \setminus X}(r)$ there exists a cop choice vertex $c_j \in \{c_1,c_2\}$ such that $v \in V_{\bot}(\botS(c_j))$. Thus,  $g((f(p),v))$ is defined for all $v \in \Reach{D \setminus X}(r)$. Following \cref{thisIsW} we can also assert that $\Reach{D \setminus S_j}(w) \subset \Reach{D \setminus X}(r)$ for $v \in \Reach{D \setminus X}(r) \cap V_{\bot}(\botS(c_j))$. $g((f(p),v),v)$ is a legal move since $V_{\bot}(\botS(c_j)) \cap S_j = \emptyset$. $\Reach{D \setminus S_j}(w) \subseteq V_{\bot}(\botS(c_j))$ since $v \in V_{\bot}(\botS(c_j))$ and by definition $S_j \subseteq \bag(c_j) \subseteq V_A(\botS(c_j))$. Thus, $c_j$ fulfills the conditions to be the cop choice vertex for $w$ under $f(p)$.

To find the forward facing vertices, we first consider a simple construction. For $t \in V(T), r \in V(D)$ and $X \in \Separators$ with $V_{\bot}(\botS(t)) \supseteq \Reach{D \setminus X}(r)$ we define $T_{t,r,X}$ to be the induced subgraph of $T$ on the set $\{w \in \Reach{T}(t) | V_{\bot}(\botS(w) \supseteq \Reach{D \setminus X}(r) \}$. This graph is non-empty as it always contains $t$ and it is connected as the bottom separations for any path in $T$ are laminar and thus, for every path from $t$ to another vertex $w$ with $V_{\bot}(\botS(w)) \supseteq \Reach{D \setminus X}(r)$, every vertex $v$ on that path has to fulfill $V_{\bot}(\botS(w)) \supseteq \Reach{D \setminus X}(r)$ as well.  
We now lead an induction to find a forward facing vertex for every cop initiative game state in a play in which all previous cop initiative game states had forward facing vertices. 

IB:
Let $p_s=\CIGSex{\emptyset}{v} \in \CIGS$ be a starting game state. 
As $V_{\bot}(\botS(s)) = V(D) \supseteq \Reach{D}(v)$, we can construct $T_{s,v,\emptyset}$.
Let $x$ be any sink of $T_{s,v,\emptyset}$.
We define $x$ to be the forward-facing vertex for $\CIGSex{\emptyset}{v}$.
$x$ satisfies \ref{item:FF:X_is_above} as $\emptyset \subseteq V_A(\botS(x))$.
$x$ satisfies \ref{item:FF:x_blocks} as $\emptyset = \emptyset \cap \bag(x)$.
$x$ satisfies \ref{item:FF:Reach_r_under} as $x$ lies in $T_{s,v,\emptyset}$ and by definition all vertices in $T_{s,v,\emptyset}$ satisfy \ref{item:FF:Reach_r_under}.
$x$ is a sink in $T_{s,v,\emptyset}$ meaning that none of its children $c$ satisfy the condition $V_{\bot}(\botS(c)) \supseteq \Reach{D}(v)$. Since $V_{\bot}(\botS(x)) \supseteq \Reach{D}(v)$, it follows that $\bag(c)$ intersects $\Reach{D}(v)$.
Thus $x$ satisfies \ref{item:FF:children_intersect}.

IH:
Let $p_{-2} \in \CIGS$ be a cop initiative game state in a play such that there exists a designated forward-facing vertex $x_{-2}$ for $p_{-2}$.

IS:
Let $p=\CIGSex{X}{v} \in \CIGS$ be a non-starting game state in a play $\pi$ such that $p_{-2} \in \CIGS$ is the last cop initiative game state before $p$ in $\pi$, $p_{-1} = f(p_{-2}) \in \RIGS$ is the robber initiative game state directly before $p$ in $\pi$ and if $p_{-1} \in \RIGS_2$, then $X = g((p_{-1},v))$. Let $t$ be the cop choice vertex for $v$ under $p_{-1}$. Note that there exists a cop choice vertex for every game state the robber can move to.
Then by definition of the cop choice vertices we know that $V_{\bot}(\botS(t)) \supseteq \Reach{D \setminus X}(v)$. So we can construct $T_{t,v,X}$. Let $x$ be any sink of $T_{t,v,X}$. We define $x$ to be the forward-facing vertex for $\CIGSex{X}{v}$.
$x$ satisfies \ref{item:FF:X_is_above} as $X \subseteq V_A(\botS(t)) \subseteq V_A(\botS(x))$.
$x$ satisfies \ref{item:FF:x_blocks} as by definition of $T_{t,v,X}$, it holds that $\Reach{D \setminus X}(v) \subseteq V_{\bot}(\botS(x))$ and thus all outgoing edges from $\Reach{D \setminus X}(v)$ end in $X \cap V_{S}(\botS(x)) \subseteq X \cap \bag(x)$.
$x$ satisfies \ref{item:FF:Reach_r_under} as $x$ lies in $T_{t,v,X}$ and by definition all vertices in $T_{t,v,X}$ satisfy \ref{item:FF:Reach_r_under}.
$x$ is a sink in $T_{t,v,X}$ meaning that none of its children $c$ satisfy the condition $V_{\bot}(\botS(c)) \supseteq \Reach{D \setminus X}(v)$. Since $V_{\bot}(\botS(x)) \supseteq \Reach{D}(v)$, it follows that $\bag(c)$ intersects $\Reach{D \setminus X}(v)$.
Thus $x$ satisfies \ref{item:FF:children_intersect}.

Since we can assign a forward facing vertex to every cop initiative game state $p$ consistent with $(f, g)$ and define $f(p)$ for all cop initiative game states consistent with $(f, g)$, we know that $(f, g)$ is complete and thus there exists a legal move after every cop initiative game state, meaning that no play consistent with $(f, g)$ ends on a cop initiative game state.
Since every play is finite by \cref{finitePlay}, $(f, g)$ is a winning strategy for the cop player.
\end{proof}}

\begin{theorem}
The cop player has a winning strategy in the $\Separations$-DAG cops-and-robber game of size $k$ on $D$ if and only if there exists an $\Separations$-DAG of width $k$ for $D$.
\end{theorem}

\begin{proof}
    This follows directly from \cref{theoStr->Dec} and \cref{theoDec->Str}.
\end{proof}

\begin{lemma}\ifshort{[\appsymb]}{}\else{}\fi
    There is no play $\pi$ of the $\Separations$-DAG cops-and-robber game which contains the same game state twice.
    \label{noStateTwice}
\end{lemma}
\appendixproof{noStateTwice}{
\begin{proof}
    Assume toward contradiction, that there exists a play $\pi$ of the $\Separations$-DAG cops-and-robber game which contains the same position twice. As laid out in \cref{proFinitePlay}, the reach of the robber in consecutive cop initiative game states is strictly decreasing. Since the reach of the robber cannot increase when moving between a cop initiative game state and a robber initiative game state, the reach of the robber between consecutive robber initiative game states is also strictly decreasing. This contradicts the existence of $\pi,$ since containing the same position twice would mean that there are two game states of the same type with the same reach for the robber in $\pi,$ which contradicts the reach of the robber decreasing strictly.
\end{proof}}

\begin{theorem}
Given a digraph $D,$ a winning strategy for the $\Separations$-DAG cops-and-robber game of size $k,$ can be computed in $O(|V(D)|^{h(k)})$ for a function $h.$
\label{compWinStr}
\end{theorem}

\ifshort{\begin{proof}[Proofsketch]
    We build a game graph for the $\Separations$-DAG cops-and-robber game where there is an edge $(p,p')$ between two game states if $p'$ is a legal move from $p$. 
    The set of all game states has size $O(|V(D)|^{h(k)})$. Since no game state can be contained twice in the same play of a $\Separations$-DAG cops-and-robber game, the game graph is acyclic. We leverage this by figuring out for every game state which player it is winning for, starting with the sinks. This can be done in polynomial time as which player is winning from a position can be identified by simply checking which outgoing neighbours are winning for which player. Doing this we can also compute winning strategies for every game state by combining the winning moves at that game state with the winning strategies from potential follow up game states.
\end{proof}}{}\else{}\fi
\appendixproof{compWinStr}{
\begin{proof}
    Given a Digraph $D,$ we construct a graph $G_D = (\mathcal{G},M),$ representing the $\Separations$-DAG cops-and-robber game of size $k.$
    Herein $\mathcal{G}$ is the set of game states in the $\Separations$-DAG cops-and-robber game of size $k$ on $D.$ Furthermore, we define $M = \{(p,p')| p, p' \in \mathcal{G} $ and moving from $ p $ to $ p' $ is a legal move.$\}$ as the relation containing all possible moves in the game. Since $\mathcal{G}$ can be constructed directly by iteration and 
    \begin{equation*}
        |\mathcal{G}| = |\mathcal{CG}| + |\mathcal{RG}_1| + |\mathcal{RG}_2| = |V(D)|^{k+1} + |V(D)|^{2 \cdot k+1} + |V(D)|^{3 \cdot k+1} \in O(|V(D)|^{3 \cdot k+1}),
    \end{equation*}
    while every element in $\mathcal{G}$ has size at most $3 \cdot k+1,$ constructing $\mathcal{G}$ can be done in $O(|V(D)|^{3 \cdot k+1}  \cdot  k).$ 
    To construct $M$ we first calculate the reach for any combination of $v \in D$ and $X \in \Separators.$ Since calculating the reach for a given graph $D'$ and $v \in D'$ can be done in $O(|V(D')|+|E(D')|) \subseteq O(|V(D')|^2),$ calculating this for any such combination can be done in $O(|V(D)|^{k+1} \cdot |V(D)|^{2}) = O(|V(D)|^{k+3}).$
    We then need to check, for every game state, which game states the players can move to. 
    
    For a robber initiative game state $p = \RIGSSpe{X}{r}{P},$ the players can only move to a position $p' = \CIGSex{S}{v} \in \CIGS$ where $S \in P.$ Since $|P| \leq 2,$ there are at most $2 \cdot |V(D)|$ game states for which we have to check whether they are legal moves from $p.$ To check whether $p'$ is a legal move from $p$ the only requirements are that $v \in \Reach{D \setminus X}(r)$ and that $\Reach{D \setminus S}(v) \subset \Reach{D \setminus X}(r).$ This can be done in $O(|V(D)|)$ for both, yielding a total time complexity in $O(2 \cdot |V(D)| \cdot |V(D)| \cdot (|V(D)|^{2 \cdot k+1} + |V(D)|^{3 \cdot k+1})) \in O(|V(D)|^{3 \cdot k+3})$ to calculate all moves starting in robber initiative game states. 

    For a cop initiative game state $p = \CIGSex{X}{r},$ the cop player can move to a position $p' = \RIGSSpe{X_{\bot}}{r}{P} \in \CIGS$ where $P \in [\Separators]^{=1} \cup [\Separators]^{=2}.$
    Since $X_{\bot}$ is fully determined through $P$ and $X,$ this leaves less than $|V(D)|^{2 \cdot k}$ possible game states for which we have to check whether they are legal moves. The requirements to check whether $p'$ is a legal move from $p$ can be checked in $O(|V(D)^2),$ given that we have already calculated all relevant reach values. Thus, the time complexity of calculating all moves starting in cop initiative game states is in $O(|V(D)|^{2 \cdot k} \cdot |V(D)|^2 \cdot |V(D)|^{k+1}) \in O(|V(D)|^{3 \cdot k+3}).$ 
    This yields the result that the total time complexity to calculate $G_D$ is in $O(|V(D)|^{3 \cdot k+3}).$

    Since the reach of the robber is decreasing after every cop move, as laid out in the proof of \cref{noStateTwice}, no walk in $G_D$ can contain the same game state twice. It follows that $G_D$ does not contain cycles and thus is a DAG. We can leverage this to calculate a winning strategy for every vertex in $G_D$ if it exists because:

    \begin{itemize}
        \item Every sink of the graph is winning for the cop player if and only if it is a robber initiative game state with a trivial strategy.
        \item For cop initiative game states, they are winning if and only if the cop player can move to a game state which is winning for them. This can be calculated by checking whether there exists such a game state as an outgoing neighbour and expanding the winning strategy from that neighbour by adding the move to this neighbor.
        \item For robber initiative game states, to identify, whether the robber player is winning, we check whether there exists any position $p'$ the robber player can choose such that no matter which legal potential cop move the cop player plays, the resulting game state is winning for the robber player. If no such $p'$ exists, we construct a winning strategy for the cop player by choosing a winning position to move to for every potential robber choice and adding these moves to the winning strategies from all resulting follow-up game states.
    \end{itemize}
    Thus, by constructing the game graph and identifying the winning player for all game states, we can detect whether the cop player has a winning strategy by checking if all starting position are winning for the cop player. and if so we have a winning strategy for the cop player by combining the winning strategies from all starting positions.
\end{proof}}

\compSDAG*

\begin{proof}
    First, by using the construction for \cref{compWinStr}, we can derive a winning strategy $(f,g)$ for the cop player in the $\Separations$-DAG cops-and-robber game of size $k$ in $O(|V(D)|^{3 \cdot k + 3}).$
    Then by using the construction for \cref{theoStr->Dec}, we can construct an $\Separations$-DAG of width $k.$
    The computation time for this construction is $O(|G_D|),$ since the size of the $\Separations$-DAG built in steps~\cref{step:1,step:2,step:3} of the construction is bounded by the size of the game graph and the size of the construction in steps~\cref{step:4,step:5} is bounded by $O(|V(D)|^{2 \cdot k + 1} \cdot |V(D)|) \subseteq O(|G_D|).$
    Since every vertex, edge and $\sigma$ value of the $\Separations$-DAG can be calculated in $O(|V(D)|^2)$ this yields a total construction time in $O(|V(D)|^2 \cdot (|V(D)|^{3k+2})) = O(|V(D)|^{3 \cdot k + 4}).$
\end{proof}

\section{Solving Parity games}
\label{sec:ParGames}
\appendixsection{sec:ParGames}

Having taken the first hurdle toward the algorithmic applicability of $\Separations$-DAGs by proposing a way to calculate an $\Separations$-DAG of minimum width for a digraph $D,$ we proceed by proposing an algorithm for solving Parity games on $\Separations$-DAGs. For a digraph $D=(V,E)$ and parity game $P=(V,E,V_0,\Omega),$ we use \cref{compSDAG} and the construction in \cref{niceExists} to obtain a nice $\Separations$-DAG $(T, \sigma)$ of width $k$ for $D$ with size in $O(|V(D)|^{h(k)})$ in $O(|V(D)|^{h(k)})$ time.
The algorithm we propose uses the algorithm for parity games on DAG decompositions from \cite{berwanger2006dag} to handle the bulk of the computation.

Since $X_d$ represents the bag of a vertex $d$ in DAG decompositions, we define $X_d = \bag(d)$ for $d \in V(T)$ as well as its derivatives $X_{\succcurlyeq d} = \bigcup_{d \preccurlyeq_T d'} X_{d'}$ and $V_d = X_{\succcurlyeq d} \setminus X_d.$
The central data structure for the algorithm on DAG decompositions is $\textit{Frontier}_{P}(t)$ where $t$ is a vertex in the decomposition. This structure contains for all combinations of starting vertex $v$ and Even player strategy $f$, a tuple containing $v$ and all Odd-optimal outcomes of plays consistent with $f$ when our scope is restricted to $V_t$.
Note that the strategies restricted to $V_t$ may result in infinite plays in $V_t$ or partial plays ending in $X_t$.

Let $R_{P, t, opt}$ be the set of tuples $(v,f)$ such that $v \in V_t,$ $f$ is a strategy and $f$ is optimal for the Even player in $V_t$ when starting in $v$ with respect to a suitable measure.
For our purposes we consider an adjusted version of $\textit{Frontier}_{P}(t),$ called $\textit{Frontier}_{P, R_{P, t, opt}}(t)$, which only contains results of strategies for the Even player that are in $R_{P, t, opt}$ and we calculate this structure for every vertex in $T$.
\ifshort{Concrete definitions for the structures used here and the proof for why we can restrict our focus to only specific strategies can be found in the appendix.}{}\else{}\fi

\toappendix{
To prove the correctness of our approach we transfer notation and data structures used by \cite{berwanger2006dag} to $\Separations$-DAGs. We then proceed by showing how these adapted data structures can be computed in $\Separations$-DAGs.

Since $X_d$ represents the bag of a vertex $d$ in DAG decompositions, we define $X_d = \bag(d)$ for $d \in V(T)$ as well as its derivatives $X_{\succcurlyeq d} = \bigcup_{d \preccurlyeq_T d'} X_{d'}$ and $V_d = X_{\succcurlyeq d} \setminus X_d.$ We also use the relation $\sqsubseteq$ as a linear ordering of the priorities by how favourable they are to Even. Thus $i \sqsubseteq j$ if
\begin{itemize}
    \item $i$ is odd and $j$ is even or
    \item $i$ and $j$ are even and $i \geq j$ or
    \item $i$ and $j$ are odd and $i \leq j.$
\end{itemize}
Furthermore, we need data structures which represent what happens when the Even player and the Odd player play according to specific strategies.

First, for a partial play $\pi$ of a parity game $P,$ we define the outcome of $\pi,$ written $\textit{out}_{P}(\pi) \in \{\textit{winOdd}, \textit{winEven}\} \cup \{(v,p) \mid v \in V(P), p \in \omega\},$ as follows
\begin{align*}
    \textit{out}_{P}(\pi) ={} \begin{cases}
                            \textit{winOdd} & \text{if } \pi \text{ is a play and is winning for Odd.} \\
                            \textit{winEven} & \text{if } \pi \text{ is a play and is winning for Even.} \\
                            (v_{i},p) & \text{if } \pi = v_0, . . . , v_i \text{ and } p = \text{min}\{\Omega(v_j) \mid j \leq i\}.
                         \end{cases} 
\end{align*}
We consider this the outcome of $\pi$ as in case $\pi$ has a winner it records that winner and in case the play is finite, it contains the relevant information for further analysis of the winner, namely the ending vertex and the lowest priority occurring within $\pi$ as this dictates which player would win if this part of the play were to occur infinitely often.
For these outcomes we use a partial ordering $\trianglelefteq$ which ranks them according to how favourable they are to Even.
The unique maximum outcome for Even ranked by $\trianglelefteq$ naturally is \textit{winEven} and the unique minimum outcome is \textit{winOdd}.
For other results $(v_1,p_1)$ and $(v_2,p_2)$ they are only comparable if $v_1 = v_2$ as we do not know how favourable different ending vertices are to Even. In the case where $v_1 = v_2$ we define $(v_1,p_1) \trianglelefteq (v_2,p_2)$ if $p_1 \sqsubseteq p_2.$ Furthermore we also introduce the non-reflexive version $\triangleleft$ of $\trianglelefteq$ which for outcomes $o_1$ and $o_2$ contains $(o_1,o_2)$ if $o_1 \trianglelefteq o_2$ and $o_1 \neq o_2.$
Since we progressively traverse the graph and try to keep our calculations as local as possible in every step, whenever we consider plays of parity games we only consider the part of those plays which takes place in the part of the graph we have already traversed. To this end we use the notation $\textit{out}_{P,f,g}(U,v)$ which describes the outcome of the partial play $\pi$ in $U$ starting in $v \in U$ for a parity game $P$ which is consistent with Even playing $f$ and with Odd playing $g$ where $U$ is the part of the graph we have traversed so far. To define this notion, let $\pi'$ be the maximal initial segment of $\pi$ which is contained in $U.$ We define $\textit{out}_{P,f,g}(U,v)$ as the outcome of $\pi'$ if $\pi'$ is infinite and as the outcome of $\pi':w$ if $\pi'$ is finite where $w$ is the first vertex after $\pi'$ in $\pi.$
Intuitively, this means that this notion represents the outcome of the partial play until it leaves our scope $U.$

From this notion which only considers fixed strategies for both players we construct a higher level notion which only fixes a strategy for the Even player and then considers the possible result when the Odd player plays any strategy. This notion, which we call $\textit{result}_{P,f}(U,v),$ is composed of the outcomes of plays in the parity game $P$ starting in $v \in U$ and contained in $U$ for which Even plays according to $f$ which are most favourable for Odd. Since being favourable to Odd is synonymous with being unfavourable to Even with respect to $\trianglelefteq,$ we can define $\textit{result}_{P,f}(U,v)$ as follows.
\begin{equation*}
    \textit{result}_{P,f}(U,v) = \{o \mid \exists g: o = \textit{out}_{P,f,g}(U,v) \text{ and } \nexists g: \textit{out}_{P,f,g}(U,v) \trianglelefteq o \}
\end{equation*}
Thus, $\textit{result}_{P,f}(U,v)$ contains up to one outcome for each vertex which has an incoming edge from a vertex in $U.$
We extend $\trianglelefteq$ and $\triangleleft$ to compare results. For two results $r_1$ and $r_2,$ we define that $r_1 \trianglelefteq r_2$ if for all outcomes $o_1 \in r_1,$ there exists an outcome $o_2 \in r_2$ with $o_1 \trianglelefteq o_2.$ As for outcomes, we say $r_1 \triangleleft r_2$ if $r_1 \trianglelefteq r_2$ and $r_1 \neq r_2$ which means that $\{\textit{winEven}\}$ is the unique maximum and $\{\textit{winOdd}\}$ is the unique minimum of $\trianglelefteq$ and $\triangleleft$ on results.
For $\textit{result}_{P,f}(U,v) \trianglelefteq \textit{result}_{P,f'}(U,v),$ we say\define{$f'$ dominates $f$ in $U$ when starting in $v$ on $P$}.

Lastly, we create a structure which no longer fixes a strategy for either player and instead considers all combinations of strategies. This structure is defined to fit our utilisation of graph decompositions.
Specifically, for a vertex $t$ in a graph decomposition, this structure, which we call $\textit{Frontier}_{P}(t),$ contains a tuple for each vertex $v \in V_t$ and strategy $f$ for the Even player, which contains $v$ and the result $\textit{result}_{P,f}(V_t,v)$ of $f$ when starting in $v.$
Explicitly, this means that $\textit{Frontier}_{P}(t)$ is defined as follows:
\begin{equation*}
    \textit{Frontier}_{P}(t) = \{(v, \textit{result}_{P,f}(V_t, v)) \mid v \in V_t \text{ and }f \text{ is a strategy}\}.
\end{equation*}

This is the central data structure of the algorithm which solves parity games on DAG decompositions.
Since the proofs henceforth contain a lot of 2-tuples, we introduce the function $\gamma$ which projects 2-tuples onto their components. We define $\gamma((a,b),1)\coloneqq a$ and $\gamma((a,b),2)\coloneqq b.$
Based on the notion of frontiers we also introduce the concept of a strategy $f$ being consistent with a set $R$ of tuples $(v,f'))$ under $t \in V(P)$ which represents $f$ behaving like a strategy $f'$ under $t$ when starting in $v.$ We say a strategy $f$ is\define{consistent with a set $R$ under $t$} if for each $v \in V(D)$ there exists $(v,f') \in R$ with $\textit{result}_{P,f}(V_t, v) = \textit{result}_{P,f'}(V_t, v).$ 
To calculate frontiers for $\Separations$-DAGs, we provide a slightly adjusted definition of the frontier data structure. Concretely, we confine frontiers to only contain outcomes for specific combinations of starting vertices and strategies. This results in the following definition for $P$ and $t \in V(T)$:

\begin{align*}
    \textit{Frontier}_{P, R}(t) \coloneqq \{ (v, \textit{result}_{P,f}(V_t, v)) \mid v \in V_t \text{ and } f \text{ is a strategy and }(v,f) \in R\}.
\end{align*}

Let $R_{P, t, opt}$ be the set of tuples $(v,f)$ such that $v \in V_t,$ $f$ is a strategy and there exists no $f'$ with $\textit{result}_{P,f}(V_t, v) \triangleleft \textit{result}_{P,f'}(V_t, v)$ and let $R_{P, t, all}$ be the set of all tuples $(v,f)$ such that $v \in V_t,$ $f$ is a strategy. }
We use the following abstraction from the algorithm for parity games on DAG decompositions:

\begin{theorem}[\cite{berwanger2006dag}]
    For a parity game $P=(V,E,W_0,\Omega),$ a nice DAG decomposition $(T,(X_t)_{t \in V(T)})$ of $(V,E),$ $d \in V(T)$ and a set $R_c$ for every $c\in \outN{D}{d}$, let $R_d$ contain all tuples of vertices $v \in V_t$ and strategies $f$ such that $f$ is consistent with the tuples in $R_c$ under $c.$
    Given $\textit{Frontier}_{P, R_c}(c)$ for all $c\in \outN{D}{d}$, the set $\textit{Frontier}_{P, R_d}(d)$ can be computed in $O(|V|^{h(k)}),$ where $h$ is a function and $k$ is the maximum among the size of the bags of $d$ and its children. \label{berwCompFront}
\end{theorem}

\toappendix{
Our algorithm calculates $\textit{Frontier}_{P, R_{P, t, opt}}(t)$ for all $t \in V(T).$ To do so we show that there is no incentive for the Even player to play a strategy which is worse than another strategy with respect to $\trianglelefteq.$ To show this we use the following lemma. 

\begin{lemma}
    \label{worseStratsDontHelp}
    Let $\pi,$ $\pi'$ and $\phi'$ be partial plays of the parity game on $P$ such that $\pi$ and $\pi'$ are finite, start in the same vertex and $\textit{out}_P(\pi) \trianglelefteq \textit{out}_P(\pi').$ Let $\phi$ be the partial play resulting from replacing any number of segments in $\phi'$ that are equal to $\pi'$ by $\pi.$ It holds that $\textit{out}_P(\phi) \trianglelefteq \textit{out}_P(\phi').$ x
\end{lemma}

\begin{proof}
    Let $p$ be the lowest priority in $\pi$ and let $p'$ be the lowest priority in $\pi'.$
    We distinguish the following two cases.

    Case 1 $\phi$ is infinite: Let $q$ be the lowest priority occurring infinitely often in $\phi,$ let $q'$ be the lowest priority occurring infinitely often in $\phi'.$ 
    If $q = q',$ we know that $\textit{out}_P(\phi) = \textit{out}_P(\phi').$
    If $\pi$ gets replaced a finite amount of times, the same priorities that occur infinitely often in $\phi$ also occur infinitely often in $\phi'$ and thus $q=q'.$ 
    If $\pi$ gets replaced infinitely often and $q \neq q',$ we know that $q=p$ or $q'=p'.$ If $q < q',$ then $p < p'$ and, since $\textit{out}_P(\pi) \trianglelefteq \textit{out}_P(\pi'),$ we know that $p$ is odd, meaning that $\textit{winOdd} = \textit{out}_P(\phi) \trianglelefteq \textit{out}_P(\phi').$ If $q' < q,$ then $p' < p$ and, since $\textit{out}_P(\pi) \trianglelefteq \textit{out}_P(\pi'),$ we know that $p'$ is even, meaning that $\textit{out}_P(\phi) \trianglelefteq \textit{winEven} = \textit{out}_P(\phi').$ 

    Case 2 $\phi$ is finite: The last vertex of $\phi$ and $\phi'$ is always identical, we call it $w.$ Let $q$ be the lowest priority occurring in $\phi,$ let $q'$ be the lowest priority occurring in $\phi'.$ If $q = q',$ we know that $\textit{out}_P(\phi) = \textit{out}_P(\phi').$ If $q \neq q',$ we know that $q=p$ or $q'=p'.$ 
    If $q < q',$ then $p < p'$ and, since $\textit{out}_P(\pi) \trianglelefteq \textit{out}_P(\pi'),$ we know that $p$ is odd which combined with $p < p'$ means that $\textit{out}_P(\phi) = (w,p) \trianglelefteq (w,p') = \textit{out}_P(\phi').$ If $q' < q,$ then $p' < p$ and, since $\textit{out}_P(\pi) \trianglelefteq \textit{out}_P(\pi'),$ we know that $p'$ is even which combined with $p' < p$ means that $\textit{out}_P(\phi) = (w,p) \trianglelefteq (w,p') = \textit{out}_P(\phi').$
\end{proof}

From \cref{worseStratsDontHelp}, it follows that there is never an advantage for the Even player in playing a strategy $f$ under a vertex $d \in V(T)$ whose result is worse then the result of another strategy $f'$ under $d$ with respect to $\trianglelefteq,$ since, for every play $\pi'$ under $d$ consistent with $f',$ there exists a play $\pi$ consistent with $f$ under $d,$ such that $\textit{out}_P(\pi) \trianglelefteq \textit{out}_P(\pi').$ }
We use \cref{berwCompFront} to calculate the frontiers for the vertices in $(T,\sigma)$ by constructing DAG decompositions with similar frontiers.

\begin{theorem}
    For a parity game $P=(V,E,W_0,\Omega),$ a nice $\Separations$-DAG $(T,\sigma)$ of $D=(V,E)$ and $d \in V(T)$ with $\outD{T}{d} \leq 1,$ given $\textit{Frontier}_{P, R_{P, c, opt}}(c)$ if $d$ has a child $c,$ $\textit{Frontier}_{P, R_{P, d, opt}}(d)$ can be computed in $O(|V|^{h(k)}),$ where $h$ is a function and $k$ is the maximum among the size of the bags of $d$ and its children. \label{degOneCompFront}
\end{theorem}
\ifshort{
\begin{proof}[Proofsketch]
    The case for $d$ being a sink is trivial. If $d$ has an outgoing neighbour we construct a DAG decomposition $(T_d,(X_t)_{t \in V(T)})$ for $(V,E).$
    \begin{align*}
        T_d &{}\coloneqq (\{d'_{\top},d',c',c'_{\bot}\}, \{(d'_{\top},d'),(d',c'),(c',c'_{\bot})\}) \\
        X_{d'_{\top}} &{}\coloneqq V \setminus V_{d} ,\quad
        X_{d'} \coloneqq X_d ,\quad
        X_{c'} \coloneqq X_c ,\quad
        X_{c'_{\bot}} \coloneqq X_{\succcurlyeq c}
    \end{align*}    
    This construction and \cref{berwCompFront} yield $\textit{Frontier}_{P, R_{P, d, opt}}(d) = \textit{Frontier}_{P, R_{P, d', opt}}(d')$.
\end{proof}}{}\else{}\fi
\appendixproof{degOneCompFront}{
\begin{proof}
    We consider the cases for what the out-neighbourhood for $d$ can look like in a nice $\Separations$-DAG.

    Case 1 $\outD{T}{d} = 0$: Then $d$ is a sink in $T$ and we know that $V_d = \emptyset.$ Thus, $\textit{Frontier}_P(d) = \emptyset$ and we choose $S_d = \emptyset.$

    Case 2 $\outN{T}{d} = \{c\}$: In this case we construct a DAG decomposition $(T_d,(X_t)_{t \in V(T)})$ for $(V,E).$ 
    \begin{align*}
        T_d &{}\coloneqq (\{d'_{\top},d',c',c'_{\bot}\}, \{(d'_{\top},d'),(d',c'),(c',c'_{\bot})\}) \\
        X_{d'_{\top}} &{}\coloneqq V \setminus V_{d} \\
        X_{d'} &{}\coloneqq X_d \\
        X_{c'} &{}\coloneqq X_c \\
        X_{c'_{\bot}} &{}\coloneqq X_{\succcurlyeq c}
    \end{align*}
    This is a DAG decomposition as $V_{d} \subseteq X_{\succcurlyeq c}$ and thus $X_{d'_{\top}} \cup X_{c'_{\bot}} = V$ which satisfies \ref{item:D:containsAll}, the condition \ref{item:D:edgesGuard} is fulfilled as $X_c,$ $X_d$ and $X_c \cap X_d$ all guard the vertices below them in $T$ since $(T,\sigma)$ is an $\Separations$-DAG and the condition \ref{item:D:laminar} is fulfilled as $X_{\preccurlyeq c} \cap (V \setminus V_{d}) = V_{B}(\topS{c}) \cap V_{A}(\botS{d}) = X_d \cap X_c.$
    We know that $\textit{Frontier}_{P, R_{P, c, opt}}(c) = \textit{Frontier}_{P, R_{P, t, opt}}(c').$
    We use the construction in \cite[Theorem 24]{BERWANGER2012900} to obtain a nice DAG decomposition from $(T_d, (X_t)_{t \in V(T)}).$
    We use $\textit{Frontier}_{P, R_{P, t, opt}}(c')$ and \cref{berwCompFront} to calculate $\textit{Frontier}_{P, R_{d'}}(d').$ The computation time for this is in $O(|V|^{h(k)})$ where $k$ is the width of $(T,\sigma)$ since by construction all bags of vertices between $c'$ and $d'$ have size at most $\text{max}\{|X_{c'}|, |X_{d'}|\}$ which is bounded by $k.$
    By \cref{worseStratsDontHelp} we know that $R_{P, d', opt} \subseteq R_{d'} \subseteq R_{P, d', all}.$ Thus, we construct $\textit{Frontier}_{P, R_{P, t, opt}}(d')$ by removing all tuples $(v, r)$ from $\textit{Frontier}_{P, R_{d'}}(d')$ for which there exists $(v,r') \in \textit{Frontier}_{P, R_{d'}}(d')$ with $r \triangleleft r'.$ The computation time for doing this by brute force is in $O(|\textit{Frontier}_{P, R_{d'}}(d')| \cdot k) \subset O(|V|^{h(k)}).$ This proves the theorem since $\textit{Frontier}_{P, R_{P, d, opt}}(d) = \textit{Frontier}_{P, R_{P, d', opt}}(d').$
\end{proof}}

The case where a vertex $d$ in the $\Separations$-DAG has two children $c_1$ and $c_2$ whose top separations uncross into $\botS(d)$ is more complicated and thus we have to lay some groundwork first.
We construct a digraph $D',$ a parity game $P'$ and a DAG decomposition $(T',(X_t)_{t \in V(T)})$ in which there exists a vertex $d'$ for which $\textit{Frontier}_{P', R_{P', d', opt}}(d')$ can be calculated using \cref{berwCompFront} and the set $\textit{Frontier}_{P, R_{P, d, opt}}(d)$ can be calculated easily using $\textit{Frontier}_{P', R_{P', d', opt}}(d').$ For this the digraph $D'$ is constructed as a copy of $D[X_{\succcurlyeq d}]$ with the exception that for every vertex $v$ which is in both $X_{\succcurlyeq c_1}$ and $X_{\succcurlyeq c_2}$ but not in $\bag(d)$ the digraph $D'$ contains two copies of $v.$ The digraph $D'$ does not copy vertices above $d$ in $D$ as $\bag(d)$ guards the vertices below $d$ and thus the vertices above $d$ do not influence $\textit{Frontier}_{P}(d).$ A closer explanation for the construction of $D'$ can be seen in \cref{fig:DPrimeConstruction}.

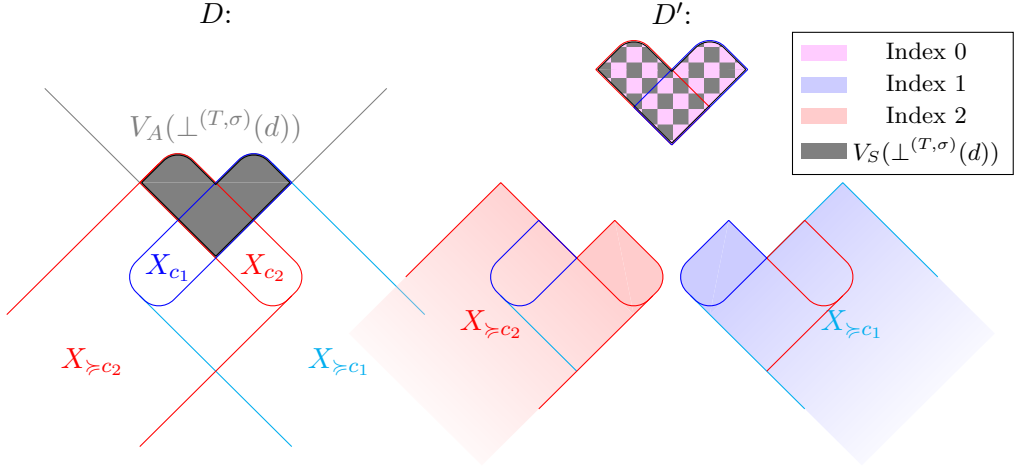
\begin{figure}[!t]
    \centering
    \begin{tikzpicture}[scale=0.5]
        \tikzfading[name=fade right, left color=transparent!0, right color=transparent!100]
        \def\topside{(0,-3.5) -- (-4.5,1)  (4.5,1) -- (0,-3.5)}
        \def\cTwo{(-2,-1.5) -- (-1,-0.5) -- (2.5,-4) -- (1.5,-5) -- (-2,-1.5)}
        \def\cTwoBot{(2,-4.5) -- (-2,-8.5) (-1.5,-1) -- (-5.5,-5)}
        \def\cOne{(2,-1.5) -- (1,-0.5) -- (-2.5,-4) -- (-1.5,-5)-- (2,-1.5)}
        \def\cOneBot{(-2,-4.5) -- (2,-8.5) (1.5,-1) -- (5.5,-5)}
        \def\XBot{(0,-3.45) -- (1.95,-1.5) (1.95,-1.5) -- (1,-0.55) -- (0,-1.55)  (0,-1.55) -- (-1,-0.55) -- (-1.95,-1.5) (-1.95,-1.5) -- (0,-3.45)}
        \def\XBotFillSharp{(2,-1.5) -- (0,-3.5) -- (-2,-1.5)}
        \def\XBotFillRounded{(2,-1.5) -- (1,-0.5) -- (0,-1.5)  (0,-1.5) -- (-1,-0.5) -- (-2,-1.5)}
        \draw (0,3) node {$D$:};
        \draw[rounded corners=8pt, color=gray] \topside;
        \draw[draw=none, fill=gray] \XBotFillSharp;
        \draw[draw=none, rounded corners=7.5pt, fill=gray] \XBotFillRounded;
        \draw[color=gray] (0,0) node {$V_A(\botS(d))$};
        \draw[color=cyan] \cOneBot (3.25,-6.25) node {$X_{\succcurlyeq c_1}$};
        \draw[color=red] \cTwoBot (-3.25,-6.25) node {$X_{\succcurlyeq c_2}$};
        \draw[rounded corners=8pt, color=darkred] \cTwo (1.25,-3.75) node {$X_{c_2}$};
        \draw[rounded corners=8pt, color=blue] \cOne (-1.25,-3.75) node {$X_{c_1}$};
        \draw[rounded corners=7.5pt, color=black] \XBot;
        \begin{customlegend}[
            legend entries={ 
                Index $0$,
                Index $1$,
                Index $2$,
                $V_S(\botS(d))$
            },
            legend style={at={(21,2.5)},font=\footnotesize}] 
            \addlegendimage{area legend,fill=lightviolet,draw opacity=0}
            \addlegendimage{area legend,fill=lightblue,draw opacity=0}
            \addlegendimage{area legend,fill=lightred,draw opacity=0}
            \addlegendimage{area legend,fill=gray,draw opacity=0}
        \end{customlegend}
        \def\XBotPrime{(13.95,1.5) -- (13,2.45) -- (12,1.45)  (12,1.45) -- (11,2.45) -- (10.05,1.5) (10.05,1.5) -- (12,-0.45) (12,-0.45) -- (13.95,1.5)}
        \def\XBotPrimeFillRounded{(13.95,1.5) -- (13,2.45) -- (12,1.45)  (12,1.45) -- (11,2.45) -- (10.05,1.5)}
        \def\XBotPrimeFillSharp{(10.05,1.5) -- (12,-0.45) -- (13.95,1.5)}
        \def\cTwoPrime{(12,-0.5) -- (13,0.5) (12,-0.5) -- (10,1.5) (10,1.5) -- (11,2.5) -- (13,0.5)
        (10.5,-2.5) -- (9.5,-3.5) (10.5,-2.5) -- (12,-4) -- (11,-5) -- (9.5,-3.5)
        (15.5,-2.5) -- (14.5,-3.5) (15.5,-2.5) -- (17,-4) -- (16,-5) -- (14.5,-3.5)}
        \def\cTwoBotPrimeFillSharpOne{(10.5,-2.5) -- (9.5,-3.5) -- (11,-5)}
        \def\cTwoBotPrimeFillSharpTwo{(3.5,-5.5) -- (7.5,-1.5) -- (11,-5) -- (7,-9)}
        \def\cTwoBotPrimeFillRound{(10.5,-2.5) -- (12,-4) -- (11,-5)}
        \def\cOneBotPrimeFillSharpOne{(13.5,-2.5) -- (14.5,-3.5) -- (13,-5)}
        \def\cOneBotPrimeFillSharpTwo{(20.5,-5.5) -- (16.5,-1.5) -- (13,-5) -- (17,-9)}
        \def\cOneBotPrimeFillRound{(13.5,-2.5) -- (12,-4) -- (13,-5)}
        \def\cTwoBotPrime{(16.5,-4.5) -- (14.5,-6.5)
        (11.5,-4.5) -- (8.5,-7.5) (7.5,-1.5) -- (5,-4) (7.5,-1.5) -- (9.5,-3.5)}
        \def\cOnePrime{(12,-0.5) -- (11,0.5) (12,-0.5) -- (14,1.5) (14,1.5) -- (13,2.5) -- (11,0.5) 
        (13.5,-2.5) -- (14.5,-3.5) (13.5,-2.5) -- (12,-4) -- (13,-5) -- (14.5,-3.5)
        (8.5,-2.5) -- (9.5,-3.5) (8.5,-2.5) -- (7,-4) -- (8,-5) -- (9.5,-3.5)}
        \def\cOneBotPrime{(7.5,-4.5) -- (9.5,-6.5) 
        (12.5,-4.5) -- (15.5,-7.5) (19,-4) -- (16.5,-1.5) (16.5,-1.5) -- (14.5,-3.5)}
        \draw (12,3) node {$D'$:};
        \draw[rounded corners=7.5pt, draw=none, fill=lightblue] \cOneBotPrimeFillRound;
        \draw[draw=none, fill=lightblue] \cOneBotPrimeFillSharpOne;
        \draw[draw=none, fill=lightblue, path fading=fade right, fading transform={rotate=-45}] \cOneBotPrimeFillSharpTwo;
        \draw[rounded corners=7.5pt, draw=none, fill=lightred] \cTwoBotPrimeFillRound;
        \draw[draw=none, fill=lightred] \cTwoBotPrimeFillSharpOne;
        \draw[draw=none, fill=lightred, path fading=fade right, fading transform={rotate=-135}] \cTwoBotPrimeFillSharpTwo;
        \draw[draw=none, fill=gray] \XBotPrimeFillSharp;
        \draw[draw=none, rounded corners=7.5pt, fill=gray] \XBotPrimeFillRounded;
        \draw[draw=none, pattern=checkerboard, pattern color=lightviolet] \XBotPrimeFillSharp;
        \draw[draw=none, rounded corners=7.5pt, pattern=checkerboard, pattern color=lightviolet] \XBotPrimeFillRounded;
        \draw[color=cyan] \cOneBotPrime (16.75,-5.25) node {$X_{\succcurlyeq c_1}$};
        \draw[color=red] \cTwoBotPrime (7.25,-5.25) node {$X_{\succcurlyeq c_2}$};
        \draw[rounded corners=8pt, color=red] \cTwoPrime;
        \draw[rounded corners=8pt, color=blue] \cOnePrime;
        \draw[rounded corners=7.5pt, color=black] \XBotPrime;
    \end{tikzpicture}
    \caption{We illustrate the construction of $D'.$
    The vertices in $D'$ are tuples containing the vertex from $D$ they represent as well a number between $0$ and $2.$
    For a vertex $v \in V(D'),$ we say $\gamma(v,1)$ is its corresponding vertex in $D$ and $\gamma(v,2)$ is its index.
    The index of a vertex is $0$(marked in violet) if it is in the bag of $d,$ $1$(marked in blue) if it is below $c_1$ and $2$(marked in red) if it is below $c_2.$ Index $1$ vertices do not have outgoing edges to index $2$ vertices and vice versa except for the index $1$ vertices which are in the separator of $c_1$ and the index $2$ vertices which are in the separator of $c_2.$}\label{fig:DPrimeConstruction}
\end{figure}
Formally $D'$ is defined as $D'\coloneqq (V', E')$ with
\begin{align*}
    V'\coloneqq{}& \{(v, 0) \mid v \in V_{S}(\botS(d))\} 
    \cup \{(v, 1) \mid v \in X_{\succcurlyeq c_1} \setminus \bag(d)\} 
    \cup \{(v, 2) \mid v \in X_{\succcurlyeq c_2} \setminus \bag(d)\} \\
    E'\coloneqq{}& \{((u, i),(v, 0)) \in V' \times V' \mid (u,v) \in E(D), i \in [0,2] \} \\
    {}\cup{}& \{((u, i),(v, i)) \in V' \times V' \mid (u,v) \in E(D), i \in [1,2]\} \\
    {}\cup{}& \{((u, i),(v, j)) \in V' \times V' \mid (u,v) \in E(D), i+j = 3, u \in X_{c_i}\}.
\end{align*}
Note that for every edge $(u,v)$ in $D[X_{\succcurlyeq d}]$ and $(u,i) \in V',$ there exists at least one edge $((u,i),(v,j))$ for a $j \in [0,2].$
We extend $D'$ to a parity game by mimicking $P.$
\begin{equation*}
    P'\coloneqq (V', V_0', E', \Omega'), \quad
    V_0'\coloneqq \{(v, i) \mid v \in V_0, (v, i) \in V'\}, \quad
    \Omega'((v,i)) \coloneqq \Omega(v)
\end{equation*}
In order to calculate frontiers for $P',$ we construct a DAG decomposition $(T',(X_t)_{t \in V(T)})$ for $D'.$
\begin{align*}
    V(T') &{}= \{d',d_{\bot}',c_1',c_2',c_{1,\succcurlyeq}',c_{2,\succcurlyeq}'\}\\
    E(T') &{}= \{(d',d_{\bot}'),(d_{\bot}',c_1'),(d_{\bot}',c_2'),(c_1',c_{1,\succcurlyeq}'),(c_2',c_{2,\succcurlyeq}')\} \\
    X_{d'} &{}= \{(u,0) \mid u \in V_{S}(\botS(d))\} \\
    X_{d_{\bot}'} &{}= \{(u,0) \mid u \in V_{S}(\botS(d)))\} \cup \{(v,1) \mid v \in X_{c_1} \setminus \bag(d)\} \cup \{(w,2) \mid w \in X_{c_2} \setminus \bag(d)\} \\
    X_{c_1'} &{}= \{(v,1),(v',0) \in V' \mid v,v' \in X_{c_1}\},\quad
    X_{c_2'} = \{(w,2),(w',0) \in V' \mid w,w' \in X_{c_2}\} \\
    X_{c_{1,\succcurlyeq}'} &{}= \{(v,1) \mid v \in X_{\succcurlyeq c_1} \setminus \bag(d)\} \cup \{(v',0) \mid v' \in X_{\succcurlyeq c_1} \cap \bag(d)\} \\
    X_{c_{2,\succcurlyeq}'} &{}= \{(v,2) \mid v \in X_{\succcurlyeq c_2} \setminus \bag(d)\} \cup \{(v',0) \mid v' \in X_{\succcurlyeq c_2} \cap \bag(d)\}
\end{align*}

\toappendix{
\begin{figure}
    \centering
    \scalebox{0.6}{\begin{tikzpicture}[>={Stealth[width=3mm,length=2mm]}]
        \begin{customlegend}[
            legend entries={ 
                Index $0$,
                Index $1$,
                Index $2$
            },
            legend style={at={(4,1.5)},font=\small}] 
            \addlegendimage{area legend,fill=lightviolet}
            \addlegendimage{area legend,fill=lightblue}
            \addlegendimage{area legend,fill=lightred}
        \end{customlegend}
        \begin{scope}
            \clip (-1,-3) rectangle (1,-2);
            \draw[draw=none, fill=lightviolet] (0,-3) circle (1cm);
        \end{scope}
        \begin{scope}
            \clip (-1,-4) rectangle (0,-3);
            \draw[draw=none, fill=lightred] (0,-3) circle (1cm);
        \end{scope}
        \begin{scope}
            \clip (0,-4) rectangle (1,-3);
            \draw[draw=none, fill=lightblue] (0,-3) circle (1cm);
        \end{scope}
        \begin{scope}
            \clip (-4,-6) rectangle (-2,-5);
            \draw[draw=none, fill=lightviolet] (-3,-6) circle (1cm);
        \end{scope}
        \begin{scope}
            \clip (-4,-7) rectangle (-2,-6);
            \draw[draw=none, fill=lightred] (-3,-6) circle (1cm);
        \end{scope}
        \begin{scope}
            \clip (4,-6) rectangle (2,-5);
            \draw[draw=none, fill=lightviolet] (3,-6) circle (1cm);
        \end{scope}
        \begin{scope}
            \clip (4,-7) rectangle (2,-6);
            \draw[draw=none, fill=lightblue] (3,-6) circle (1cm);
        \end{scope}
        \draw[fill=lightviolet] (0,0) circle (1.5cm) node {$V_{S}(\botS(d)$};
        \draw[] (0,-3) circle (1cm) node {$X_{c_2} \cup X_{c_1}$};
        \draw[] (3,-6) circle (1cm) node {$X_{c_1}$};
        \draw[] (-3,-6) circle (1cm) node {$X_{c_2}$};
        \draw[fill=lightblue] (3,-9) circle (1cm) node {$X_{\succcurlyeq c_1}$};
        \draw[fill=lightred] (-3,-9) circle (1cm) node {$X_{\succcurlyeq c_2}$};
        \begin{scope}
            \clip (-4,-8.5) rectangle (-2,-8);
            \draw[draw=none, fill=lightviolet] (-3,-9) circle (1cm);
        \end{scope}
        \begin{scope}
            \clip (2,-8.5) rectangle (4,-8);
            \draw[draw=none, fill=lightviolet] (3,-9) circle (1cm);
        \end{scope}
        \path[->] (0,-1.5) edge (0,-2)
        (-0.71,-3.71) edge (-2.29,-5.29)
        (0.71,-3.71) edge (2.29,-5.29)
        (-3,-7) edge (-3,-8)
        (3,-7) edge (3,-8);
    \end{tikzpicture}}
    \caption{This is an illustration of the DAG decomposition $(T',(X_t)_{t \in V(T)})$ for $D'.$ Note that the bags of $X_{c_1'}$ and $X_{c_2'}$ contain only the vertices matching their index and thus the vertices which correspond to vertices in $X_{c_i}$ but do not have index $i$ are not in the bag of $X_{c_i'}$ and thus also not in the bag of $X_{d_{\bot}'}.$}
    \label{fig:DPrimeDagDecompConstruction}
\end{figure}
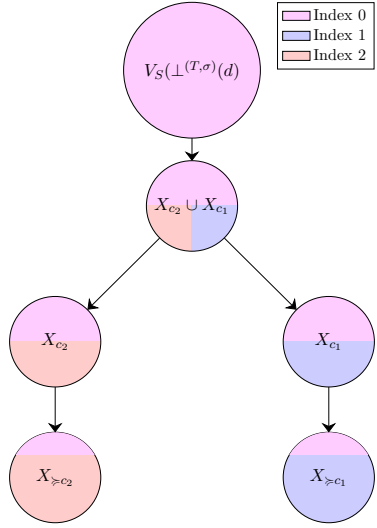

The construction $(T',(X_t)_{t \in V(T)})$ is a DAG decomposition for $D'$ as it satisfies the properties of a DAG decomposition. It satisfies \ref{item:D:containsAll} as $X_{c_{1,\succcurlyeq}'}, X_{c_{2,\succcurlyeq}'}$ and $X_{d'}$ contain all vertices of $D'$ and it satisfies \ref{item:D:laminar} as $X_{d_{\bot}'} \cap X_{c_{i,\succcurlyeq}'} = X_{c_i'},$ $X_{d'} \cap X_{c_{i,\succcurlyeq}'} \subseteq X_{c_i'}$ and $X_{d'} \cap X_{c_1'} = X_{d_{\bot}'}$ for $i \in \{1,2\}.$
We consider all edges to show $(T',(X_t)_{t \in V(T)})$ satisfies \ref{item:D:edgesGuard}. For the only source $d'$ all vertices of $D'$ are contained in bags below $d'$ thus they have no outgoing edges and are guarded by $\emptyset.$ The same goes for the edge $(d',d_{\bot}')$ as all vertices are in $X_{\succcurlyeq d_{\bot}'}.$ For $(d_{\bot}',c_i')$ with $i \in [1,2],$ the vertices in $X_{\succcurlyeq c_i'} \setminus X_{c_i'}$ do not have edges to the vertices with index $3-i$ and, since they correspond to $V_{\bot}(\topS(c_i)),$ they only have edges to vertices in $D'$ corresponding to vertices in $V_S(\topS(c_i)).$ Vertices corresponding to $V_S(\topS(c_i))$ with index $0$ or $i$ are all contained in $X_{d_{\bot}'}$ and $X_{c_i'}.$ For the edges $(c_i',c_{i,\succcurlyeq}')$ with $i \in [i,2]$ the correctness follows from the correctness for $(d_{\bot}',c_i')$ as $X_{d_{\bot}'} \cap X_{c_i'} \subseteq X_{c_i'} \cap X_{c_{i,\succcurlyeq}'}.$ 

In order to compare the vertices, outcomes, results, frontiers and plays for parity games $P$ and $P',$ we introduce the notion of congruence for all of them which defines two vertices, outcomes, results, frontiers or plays as congruent if they are identical up to the indices of the vertices involved. We say $v \in V(D)$ is congruent to $v' \in V(D'),$ written $v \equiv v',$ if $\gamma(v',1)=v.$ Furthermore, $\textit{out}_{P',f',g'}(v')$ is congruent to $\textit{out}_{P,f,g}(v)$ written $\textit{out}_{P,f,g}(v) \equiv \textit{out}_{P',f',g'}(v'),$ if 
\begin{align*}
    &\textit{winEven} = \textit{out}_{P,f,g}(v) = \textit{out}_{P',f',g'}(v') \\
    \text{or } &\textit{winOdd} = \textit{out}_{P,f,g}(v) = \textit{out}_{P',f',g'}(v') \\
    \text{or } &\begin{pmatrix}
                \gamma(\textit{out}_{P,f,g}(v),1) = \gamma(\gamma(\textit{out}_{P',f',g'}(v'),1),1) \\ 
                \text{ and } \gamma(\textit{out}_{P,f,g}(v),2) = \gamma(\textit{out}_{P',f',g'}(v'),2)
    \end{pmatrix}.
\end{align*}
That is if both outcomes are identical up to the index of the last vertex played.
Similarly, for $\textit{result}_{P,f}(U,v)$ and $\textit{result}_{P',f'}(U',v'),$ we say they are congruent or $\textit{result}_{P,f}(U,v) \equiv \textit{result}_{P',f'}(U',v'),$ if for all $o \in \textit{result}_{P',f'}(U',v'),$ there exists $o' \in \textit{result}_{P,f}(U,v)$ with $o \equiv o'$ and, for all $o \in \textit{result}_{P,f}(U,v),$ there exists $o' \in \textit{result}_{P',f'}(U',v')$ with $o \equiv o'.$ For $\textit{Frontier}_{P}(v)$ and $\textit{Frontier}_{P'}(v'),$ we say $\textit{Frontier}_{P}(v)$ is congruent to $\textit{Frontier}_{P'}(v')$ and write $\textit{Frontier}_{P}(v) \equiv \textit{Frontier}_{P'}(v'),$ if for all $(v', r') \in \textit{Frontier}_{P'}(v'),$ there exists $(v, r) \in \textit{Frontier}_{P}(v)$ with $v = \gamma(v',1)$ and $r \equiv r'$ and for all $(v, r) \in \textit{Frontier}_{P}(v),$ there exists $(v', r') \in \textit{Frontier}_{P'}(v')$ with $v = \gamma(v',1)$ and $r \equiv r'.$
Additionally, we define two plays $\pi$ and $\pi'$ for $P$ and $P'$ to be congruent or $\pi \equiv \pi',$ if $\pi[i] = \gamma(\pi'[i],1)$ for every $i.$
Note that for every vertex, outcome, result, frontier and play of $P'$ there exists exactly one such object for $D[X_{\succcurlyeq d}]$ which is congruent, while there can be multiple congruent objects for vertices, outcomes, results, frontiers and plays in $D[X_{\succcurlyeq d}]$ due to the two copies of vertices in $(X_{ \succcurlyeq c_1} \cap X_{ \succcurlyeq c_2}) \setminus X_d.$

Since it has been shown that, if there is a winning strategy for a player, there is also always a memoryless winning strategy for that player \cite{emerson1991tree}, we only have to consider memoryless strategies in parity games. 

\begin{lemma}
\begin{sloppypar}
    For each memoryless strategy $f'$ for the Even player in $P'$ and ${v \in V_{d'}},$ there exists a strategy $f$ for the Even player in $P$ and a result $r$ such that ${r \equiv \textit{result}_{P',f'}(V_{d'},v)}$ and $r \trianglelefteq \textit{result}_{P,f}(V_{d},\gamma(v,1)).$
\end{sloppypar} \label{DsimsDP}
\end{lemma}

\begin{proof}
    Let $f'$ be a memoryless strategy for the Even player in $P'$ and let $v \in V_{d'}.$
    For sequences $s$ of vertices in $D$ starting in $\gamma(v,1)$ and $\textit{last} \in V(D),$ we define $\zeta(s:last).$
    This function reconstructs which index the vertices in a play would have if it were played in $P'$ according to $f'.$
    We define it recursively as follows.
    
    \begin{minipage}{0.77\linewidth}
    \begin{flushleft}
    \vspace{5pt}
    \begin{align*}
        \zeta(s:last)\coloneqq\begin{cases}
            \gamma(v,2) & \text{if }s=() \text{ and }last = \gamma(v,1) \\
            0 & \text{else if }last \in V_{S}(\botS(d)) \\
            \zeta(s) & \text{else if }(\tail{s},\zeta(s)) \notin X_{c_{\zeta(s)}'} \\
            \gamma(f'(\tail{s}, \zeta(s)), 2) & \text{else if }\tail{s} \in V_0 (\text{case 4})\\
            1 & \text{else if } (last, 1) \in V' \\
            2 & \text{else} 
        \end{cases}
    \end{align*}
    \end{flushleft}
    \end{minipage}
    \hfill
    \begin{minipage}{0.2\linewidth}
    \begin{flushright}
    \begin{align*}
        \text{(case 1)}\\
        \text{(case 2)}\\
        \text{(case 3)}\\
        \text{(case 4)}\\
        \text{(case 5.1)}\\
        \text{(case 5.2)}
    \end{align*}
    \end{flushright}
    \end{minipage}
    
    The function $\zeta$ correctly mirrors the behaviour of $f'$ as 
    \begin{itemize}
        \item the index for the first vertex is the index of $v$ (case 1) and the index is $0$ if the vertex only exists with index $0$ in $D'$ (case 2).
        \item the index does not change in case the prior vertex does not have edges to vertices with the other index. This is true for all vertices which are not in $X_{c_1'}$ or $X_{c_2'}$ (case 3).
        \item the index imitates the behaviour of $f'$ in case the prior vertex in the sequence has edges to vertices of different indices and the vertex is in $V_0$ (case 4).
        \item the index is chosen arbitrarily in $V'$ in case the prior vertex is in $X_{c_1'}$ or $X_{c_2'}$ and not in $V_0$ as $f'$ does not dictate specific behaviour and the out-neighbourhood of these vertices contains all versions of the contained vertices (case 5).
    \end{itemize}
    Since $\zeta$ correctly recreates the behaviour of $f',$ it correctly predicts the index of $last.$ 
    We define $f$ for a sequence $seq:last$ of vertices in $D$ where $last \in V_0.$ 
    \begin{equation*}
        f(seq:last)\coloneqq \gamma(f'((last,\zeta(seq:last))),1)
    \end{equation*}
    We show that for every partial play $\pi$ under $d$ in $P$ consistent with $f$ starting in $\gamma(v,1),$ there exists a partial play $\pi'$ with the same outcome under $d'$ in $P'$ consistent with $f'$ starting in $v.$ 
    For a partial play $\pi,$ let $\pi'$ be the partial play obtained from $\pi$ by, for every $i \in [1,|\pi|],$ setting $\pi'[i]$ to $(\pi[i],\zeta(\pi[i])).$ This is a legal play as by the definition of $\zeta$ and given the fact that $(\pi[i],\pi[i+1]) \in E(D),$ it follows that $((\pi[i],\zeta(\pi[i])),(\pi[i+1],\zeta(\pi[i+1]))) \in E(D').$
    The partial play $\pi'$ is also consistent with $f'$ as $\pi'$ mirrors the behaviour of $\pi$ and $\pi$ is consistent with $f$ and by extension $\zeta$ which are defined through the behaviour of $f'.$
    The outcome of these plays is congruent as $\Omega(\pi[i]) = \Omega'(\pi'[i])$ and the last vertices in $\pi$ and $\pi'$ are congruent. Let $O'$ be the set of outcomes consistent with $f'$ when starting in $v$ in $P'$ and let $O$ be the set of outcomes on $P$ which are congruent to an outcome in $O'.$ As shown, for every outcome $o$ consistent with $f$ starting in $\gamma(v,1),$ there exists an outcome $o' \in O'$ with $o \equiv o'.$ Thus, for the set $r \equiv \textit{result}_{P',f'}(V_{d'},v)$ which is a subset of $O$ it holds that $r \trianglelefteq \textit{result}_{P,f}(V_{d},\gamma(v,1)).$ 
\end{proof}

\begin{lemma}
    For each memoryless strategy $f$ for the Even player in $P,$ $v \in V_d$ and $i \in [1,2],$ there exists a memoryless strategy $f'$ for the Even player in $P'$ such that $\textit{result}_{P,f}(V_d,v) \equiv \textit{result}_{P',f'}(V_{d'},(v,i)).$
    \label{DPsimsD}
\end{lemma}

\begin{proof}
    Let $f$ be a memoryless strategy for the Even player in $P$ and let $v \in V_d.$ 
    We define a memoryless strategy $f'$ in $P'$ for a sequence of vertices $s$ and $w \in V(D')$ in the following way: 
    \begin{equation*}
        f'(s:w)\coloneqq \begin{cases}
            (f(\gamma(w,1)),0) & \text{if }(w, (f(\gamma(w,1)),0)) \in E' \\
            (f(\gamma(w,1)),1) & \text{if }(w, (f(\gamma(w,1)),1)) \in E' \\
            (f(\gamma(w,1)),2) & \text{else if }(w, (f(\gamma(w,1)),2)) \in E'
        \end{cases}
    \end{equation*}
    The strategy $f'$ is legal as, for every vertex $w \in V(D')$ and $(\gamma(w,1),u) \in E(D),$ there exists at least one $u' \in V(D')$ with $\gamma(u',1)=u$ and $(w,u') \in E(D').$ 
    Following $f$ and $f'$ respectively leads to congruent plays. For every partial play $\pi'$ in $P'$ starting in $(v,i)$ which is consistent with $f',$ the play $\pi$ defined by $\pi[j] = \gamma(\pi'[j],1),$ for all $j,$ is legal and consistent with $f,$ as every edge in $D'$ has an equivalent edge in $D$ and every move from a  vertex in $V_0'$ dictated by $f'$ is identical to the move in $f$ because every move in $f'$ is defined through $f.$ The outcome of these plays is also congruent as $\Omega(\pi[j]) = \Omega(\pi'[j])$ and the last vertices in $\pi$ and $\pi'$ are congruent. Equivalently, for every partial play $\pi$ in $D$ which is consistent with $f,$ we define the partial play $\pi'$ recursively through $\pi'[0] = (v,i)$ and, for $k \in [1, |\pi|-2],$ $j \in [0,2]$ must be chosen such that $(\pi'[k-1],(\pi[k],j)) \in E(D').$ Such a $j$ always exists since for every edge $(u,w) \in E(D)$ and $(u,m) \in V'$ there exists $n$ with $((u,m),(w,n)) \in E(D').$ Furthermore, the play $\pi'$ is consistent with $f',$ as for every $w \in V_0'$ the successor in $\pi'$ is chosen according to the definition of $f',$ and the outcome of $\pi'$ is congruent to the outcome of $\pi$ as $\Omega(\pi[k]) = \Omega(\pi'[k])$ and the last vertices in $\pi$ and $\pi'$ are congruent. Since for every partial play consistent with $f$ or $f',$ there exists a congruent partial play for the other strategy, it follows that $\textit{result}_{P,f}(V_d,v) \equiv \textit{result}_{P',f'}(V_{d'},v).$
\end{proof}

\begin{lemma}
    $\textit{Frontier}_{P', R_{P', d', opt}}(d') \equiv \textit{Frontier}_{P, R_{P, d, opt}}(d).$
    \label{DPFrontAnalogDFront}
\end{lemma}

\begin{proof}

if for all $r' \in \textit{Frontier}_{P'}(v'),$ there exists $r \in \textit{Frontier}_{P}(v)$ with $\gamma(r,1) = \gamma(\gamma(r',1),1)$ and $\gamma(r,2) \equiv \gamma(r',2)$ and for all $r \in \textit{Frontier}_{P}(v),$ there exists $r' \in \textit{Frontier}_{P'}(v')$ with $\gamma(r,1) = \gamma(\gamma(r',1),1)$ and $\gamma(r,2) \equiv \gamma(r',2).$

    We know that for each $(v, r) \in \textit{Frontier}_{P, R_{P, d, opt}}(d)$ there exists $(v', r') \in \textit{Frontier}_{P'}(d')$ with $v = \gamma(v',1)$ and $r \equiv r'$ from \cref{DPsimsD}. We also know that $(v', r') \in \textit{Frontier}_{P', R_{P', d', opt}}(d')$ because any tuple $(v',r'') \in \textit{Frontier}_{P', R_{P', d', opt}}(d')$ with $r' \triangleleft r''$ would together with \cref{DsimsDP} imply that there exists $(v, r''') \in \textit{Frontier}_{P}(d)$ with $r \triangleleft r'''$ which would mean that $(v, r) \notin \textit{Frontier}_{P, R_{P, d, opt}}(d).$ 
    Since for every $(v', r') \in \textit{Frontier}_{P', R_{P', d', opt}}(d')$ there exists by definition no $(v',r'') \in \textit{Frontier}_{P'}(d')$ with $r' \triangleleft r'',$ we can use \cref{DsimsDP} to follow that there exists $(v,r) \in \textit{Frontier}_{P}(d)$ such that $v = \gamma(v',1)$ and $r \equiv r'.$ And similar to the argument above we can follow from $\cref{DPsimsD}$ that $(v,r) \in \textit{Frontier}_{P, R_{P, d, opt}}(d)$ proving the lemma.
\end{proof}}

\begin{theorem}
    For a parity game $P=(V,E,W_0,\Omega),$ a nice $\Separations$-DAG $(T,\sigma)$ of $D=(V,E)$ and $d \in V(T)$ with $\outN{T}{d} = \{c_1,c_2\},$ from $\textit{Frontier}_{P, R_{P, c_1, opt}}(c_1)$ and $\textit{Frontier}_{P, R_{P, c_2, opt}}(c_2)$ the set $\textit{Frontier}_{P, R_{P, d, opt}}(d)$ can be calculated in $O(|V(D)|^{u(k)})$ for a function $u$ where $k$ is the width of $(T,\sigma).$
    \label{FrontInTime}
\end{theorem}
\ifshort{
\begin{proof}[Proofsketch]
    The frontiers for $c_1'$ and $c_2'$ in the DAG decomposition $(T',(X_t)_{t \in V(T)})$ are the same as the frontiers $P$ has for $c_1$ and $c_2$ as their respective subgraphs are isomorphic. Given this, we can use $\textit{Frontier}_{P, R_{P, c_1, opt}}(c_1)$ and $\textit{Frontier}_{P, R_{P, c_2, opt}}(c_2)$ to calculate $\textit{Frontier}_{P', R_{P', d', opt}}(d')$ using \cref{berwCompFront}. Following this we can construct $\textit{Frontier}_{P, R_{P, d, opt}}(d)$ as it is directly analogue to $\textit{Frontier}_{P', R_{P', d', opt}}(d')$.
\end{proof}}{}\else{}\fi
\appendixproof{FrontInTime}{
 \begin{proof}
    The frontiers for $c_1'$ and $c_2'$ in the DAG decomposition $(T',(X_t)_{t \in V(T)})$ are the same as the frontiers $P$ has for $c_1$ and $c_2$ because the subgraphs of $D'$ below $c_1'$ and $c_2'$ are isomorphic to the subgraphs of $D$ below $c_1$ and $c_2$ with the isomorphism $h(v')=\gamma(v',1)$ for $v' \in X_{\succcurlyeq c_i'}$ for $i \in \{1,2\}.$ Given this, we can derive $\textit{Frontier}_{P', R_{P', c_1', opt}}(c_1')$ and $\textit{Frontier}_{P', R_{P', c_2', opt}}(c_2')$ from $\textit{Frontier}_{P, R_{P, c_1, opt}}(c_1)$ and $\textit{Frontier}_{P, R_{P, c_2, opt}}(c_2).$
    We use the construction in \cite[Theorem 24]{BERWANGER2012900} to obtain a nice DAG decomposition from $(T', (X_t)_{t \in V(T)}).$
    We then use \cref{berwCompFront} to calculate $\textit{Frontier}_{P', R_{d'}}(d').$ Since the sizes of $X_{c_1'},$ $X_{c_2'},$ $X_{d'}$ and $X_{d_{\bot}'}$ are all bounded by at most twice the width of $(T,\sigma),$ this can be done in $O(|V(D)|^{u(k)})$ where $k$ is the width of $(T,\sigma).$
    By \cref{worseStratsDontHelp} we know that $R_{P', d', opt} \subseteq R_{d'} \subseteq R_{P', d', all}.$ Thus, we construct $\textit{Frontier}_{P', R_{P', d', opt}}(d')$ by removing all tuples $(v, r)$ from $\textit{Frontier}_{P, R_{d'}}(d')$ for which there exists $(v,r') \in \textit{Frontier}_{P', R_{d'}}(d')$ with $r \triangleleft r'.$ The computation time for doing this by brute force is in $O(|\textit{Frontier}_{P', R_{d'}}(d')| \cdot k) \subset O(|V|^{u(k)}).$ 
    Since $\textit{Frontier}_{P', R_{P', d', opt}}(d') \equiv \textit{Frontier}_{P, R_{P, d, opt}}(d)$ by \cref{DPFrontAnalogDFront} we can construct $\textit{Frontier}_{P, R_{P, d, opt}}(d)$ from $\textit{Frontier}_{P', R_{P', d', opt}}(d')$ in time linear to the size of $\textit{Frontier}_{P', R_{P', d', opt}}(d').$
 \end{proof}}

\solveParity*

\begin{proof}
    Having shown that for any type of successors we can calculate $\textit{Frontier}_{P, R_{P, d, opt}}(d)$ of the predecessor $d$ in $O(|V(D)|^{u(k)}),$ we can calculate the frontier $\textit{Frontier}_{P, R_{P, s, opt}}(s)$ for the source $s$ in $O((|V(D)|^{u(k)})).$ Since $V_s = V$ there are no finite plays in $V_s$ meaning that the only possible results in $\textit{Frontier}_{P, R_{P, s, opt}}(s)$ are a win for the Odd player or a win for the Even player. Since $\textit{Frontier}_{P, R_{P, s, opt}}(s)$ only considers optimal strategies for the Even player $\textit{Frontier}_{P, R_{P, s, opt}}(s)$ contains for each $v \in V$ the tuple $(v, \textit{winEven})$ if there exists a strategy for the Even player which wins the game when starting in $v$ and contains $(v,\textit{winOdd})$ otherwise. Thus, we can directly decide the winner of the parity game on $D$ for each starting vertex and subsequently decide the winner of the parity game $P$.
\end{proof}

\bibliographystyle{alpha}
\bibliography{bibfile}

\newcommand{\etalchar}[1]{$^{#1}$}
\begin{thebibliography}{BKTW21}

\bibitem[ACP87]{arnborg1987complexity}
Stefan Arnborg, Derek~G Corneil, and Andrzej Proskurowski.
\newblock Complexity of finding embeddings in a k-tree.
\newblock {\em SIAM Journal on Algebraic Discrete Methods}, 8(2):277--284,
  1987.

\bibitem[AP89]{arnborg1989linear}
Stefan Arnborg and Andrzej Proskurowski.
\newblock Linear time algorithms for np-hard problems restricted to partial
  k-trees.
\newblock {\em Discrete applied mathematics}, 23(1):11--24, 1989.

\bibitem[BDH{\etalchar{+}}12]{BERWANGER2012900}
Dietmar Berwanger, Anuj Dawar, Paul Hunter, Stephan Kreutzer, and Jan
  Obdržálek.
\newblock The dag-width of directed graphs.
\newblock {\em Journal of Combinatorial Theory, Series B}, 102(4):900--923,
  2012.

\bibitem[BDHK06]{berwanger2006dag}
Dietmar Berwanger, Anuj Dawar, Paul Hunter, and Stephan Kreutzer.
\newblock Dag-width and parity games.
\newblock In {\em STACS 2006: 23rd Annual Symposium on Theoretical Aspects of
  Computer Science, Marseille, France, February 23-25, 2006. Proceedings 23},
  pages 524--536. Springer, 2006.

\bibitem[BJG08]{bang2008digraphs}
J{\o}rgen Bang-Jensen and Gregory~Z Gutin.
\newblock {\em Digraphs: theory, algorithms and applications}.
\newblock Springer Science \& Business Media, 2008.

\bibitem[BKTW21]{bonnet2021twin}
{\'E}douard Bonnet, Eun~Jung Kim, St{\'e}phan Thomass{\'e}, and R{\'e}mi
  Watrigant.
\newblock Twin-width i: tractable fo model checking.
\newblock {\em ACM Journal of the ACM (JACM)}, 69(1):1--46, 2021.

\bibitem[CJK{\etalchar{+}}17]{calude2017deciding}
Cristian~S Calude, Sanjay Jain, Bakhadyr Khoussainov, Wei Li, and Frank
  Stephan.
\newblock Deciding parity games in quasipolynomial time.
\newblock In {\em Proceedings of the 49th Annual ACM SIGACT Symposium on Theory
  of Computing}, pages 252--263, 2017.

\bibitem[Cou90]{courcelle1990monadic}
Bruno Courcelle.
\newblock The monadic second-order logic of graphs. i. recognizable sets of
  finite graphs.
\newblock {\em Information and computation}, 85(1):12--75, 1990.

\bibitem[DM{\v{S}}21]{dallard2021treewidth}
Cl{\'e}ment Dallard, Martin Milani{\v{c}}, and Kenny {\v{S}}torgel.
\newblock Treewidth versus clique number. ii. tree-independence number.
\newblock {\em arXiv preprint arXiv:2111.04543}, 2021.

\bibitem[EGW01]{espelage2001solve}
Wolfgang Espelage, Frank Gurski, and Egon Wanke.
\newblock How to solve np-hard graph problems on clique-width bounded graphs in
  polynomial time.
\newblock In {\em Graph-Theoretic Concepts in Computer Science: 27th
  InternationalWorkshop, WG 2001 Boltenhagen, Germany, June 14--16, 2001
  Proceedings 27}, pages 117--128. Springer, 2001.

\bibitem[EJ91]{emerson1991tree}
E~Allen Emerson and Charanjit~S Jutla.
\newblock Tree automata, mu-calculus and determinacy.
\newblock In {\em FoCS}, volume~91, pages 368--377. Citeseer, 1991.

\bibitem[GHK{\etalchar{+}}16]{ganian2016there}
Robert Ganian, Petr Hlin{\v{e}}n{\`y}, Joachim Kneis, Daniel Meister, Jan
  Obdr{\v{z}}{\'a}lek, Peter Rossmanith, and Somnath Sikdar.
\newblock Are there any good digraph width measures?
\newblock {\em Journal of Combinatorial Theory, Series B}, 116:250--286, 2016.

\bibitem[Hat23]{dissMeike}
Meike Hatzel.
\newblock {\em Dualities in graphs and digraphs}.
\newblock Universit{\"a}tsverlag der TU Berlin, 2023.

\bibitem[HK08]{hunter2008digraph}
Paul Hunter and Stephan Kreutzer.
\newblock Digraph measures: Kelly decompositions, games, and orderings.
\newblock {\em Theoretical Computer Science}, 399(3):206--219, 2008.

\bibitem[JRST01]{johnson2001directed}
Thor Johnson, Neil Robertson, Paul~D Seymour, and Robin Thomas.
\newblock Directed tree-width.
\newblock {\em Journal of Combinatorial Theory, Series B}, 82(1):138--154,
  2001.

\bibitem[KLM09]{kaminski2009recent}
Marcin Kami{\'n}ski, Vadim~V Lozin, and Martin Milani{\v{c}}.
\newblock Recent developments on graphs of bounded clique-width.
\newblock {\em Discrete Applied Mathematics}, 157(12):2747--2761, 2009.

\bibitem[KO11]{kreutzer2011digraph}
Stephan Kreutzer and Sebastian Ordyniak.
\newblock Digraph decompositions and monotonicity in digraph searching.
\newblock {\em Theoretical Computer Science}, 412(35):4688--4703, 2011.

\bibitem[MBK21]{merlin2021spined}
Benjamin Merlin~Bumpus and Zoltan~A Kocsis.
\newblock Spined categories: generalizing tree-width beyond graphs.
\newblock {\em arXiv e-prints}, pages arXiv--2104, 2021.

\bibitem[Obd06]{obdrzalek2006dag}
Jan Obdrz{\'a}lek.
\newblock Dag-width: connectivity measure for directed graphs.
\newblock In {\em SODA}, volume~6, pages 814--821, 2006.

\bibitem[OOT93]{oporowski1993typical}
Bogdan Oporowski, James Oxley, and Robin Thomas.
\newblock Typical subgraphs of 3-and 4-connected graphs.
\newblock {\em Journal of Combinatorial Theory, Series B}, 57(2):239--257,
  1993.

\bibitem[Oum05]{oum2005rank}
Sang-il Oum.
\newblock Rank-width and vertex-minors.
\newblock {\em Journal of Combinatorial Theory, Series B}, 95(1):79--100, 2005.

\bibitem[Ree99]{reed1999introducing}
Bruce Reed.
\newblock Introducing directed tree width.
\newblock {\em Electronic Notes in Discrete Mathematics}, 3:222--229, 1999.

\bibitem[RS84]{robertson1984graph}
Neil Robertson and Paul~D Seymour.
\newblock Graph minors. iii. planar tree-width.
\newblock {\em Journal of Combinatorial Theory, Series B}, 36(1):49--64, 1984.

\bibitem[RS86]{robertson1986graph}
Neil Robertson and Paul~D Seymour.
\newblock Graph minors. v. excluding a planar graph.
\newblock {\em Journal of Combinatorial Theory, Series B}, 41(1):92--114, 1986.

\end{thebibliography}


    \ifshort{}
    \newpage
    \appendix
    \appendixProofText
    \fi{}

\end{document}